\documentclass[11pt,leqno]{article}
\usepackage{amsmath, amscd, amsthm, amssymb, graphics, xypic, mathrsfs, setspace, fancyhdr, times, bm, enumitem}
\usepackage[letterpaper,top=1in, bottom=1in, left=1.05in, right=1.05in]{geometry}
\usepackage[colorlinks=true,pagebackref=true]{hyperref} 
\hypersetup{backref}

\newcommand{\tensor}{\otimes}
\newcommand{\colim}{\operatorname{colim}}
\newcommand{\Spec}{\operatorname{Spec}}

\newcommand{\isomto}{{\stackrel{\sim}{\;\longrightarrow\;}}}
\newcommand{\isomt}{{\stackrel{{\scriptscriptstyle{\sim}}}{\;\rightarrow\;}}}

\newcommand{\sma}{{\scriptstyle{\wedge}}}

\newcommand{\Singaone}{\operatorname{Sing}^{\aone}\!\!}

\renewcommand{\hom}{\operatorname{Hom}}

\newcommand{\Z}{{\mathbb Z}}
\newcommand{\N}{{\mathbb N}}

\newcommand{\aone}{{\mathbb A}^1}

\newcommand{\gm}[1]{{{\mathbf G}_{m}^{#1}}}

\newcommand{\MW}{\mathrm{MW}}

\newcommand{\ho}[1]{\mathscr{H}({#1})}
\newcommand{\hop}[1]{\mathscr{H}_{\bullet}({#1})}

\newcommand{\bpi}{\bm{\pi}}

\newcommand{\Nis}{\operatorname{Nis}}

\newcommand{\CH}{{\widetilde{CH}}}

\newcommand{\Sm}{\mathrm{Sm}}
\newcommand{\map}{\mathrm{Map}}

\newcommand{\Spc}{\mathrm{Spc}}
\newcommand{\Spt}{\mathrm{Spt}}

\newcommand{\K}{{{\mathbf K}}}

\newcommand{\F}{{\mathcal F}}

\newcommand{\id}{\mathrm{Id}}

\newcommand{\Laone}{{\mathrm L}_{\aone}}

\newcommand{\Addresses}{{
 \bigskip
 \footnotesize

 A.~Asok, Department of Mathematics, University of Southern California, 3620 S. Vermont Ave.,
  Los Angeles, CA 90089-2532, United States; \textit{E-mail address:} \url{asok@usc.edu}

  \medskip

 J.~Fasel, Institut Fourier - UMR 5582, Universit\'e Grenoble Alpes, CS 40700, F-38058 Grenoble; France; \textit{E-mail address:} \url{jean.fasel@gmail.com}
 
 \medskip
 
 M.K.~Das, Stat-Math Unit, Indian Statistical Institute, 203 B. T. Road, Kolkata 700108 India; \textit{E-mail address:} \url{mrinal@isical.ac.in}

}}

\newcounter{intro}
\theoremstyle{plain}
\newtheorem{thm}{Theorem}[subsection]

\newtheorem{lem}[thm]{Lemma}
\newtheorem{cor}[thm]{Corollary}
\newtheorem{prop}[thm]{Proposition}
\newtheorem*{claim*}{Claim}  

\newtheorem*{thm*}{Theorem}
\newtheorem*{problem*}{Problem}

\newtheorem{thmintro}{Theorem}

\theoremstyle{definition}
\newtheorem{defn}[thm]{Definition}
\newtheorem{construction}[thm]{Construction}

\theoremstyle{remark}
\newtheorem{rem}[thm]{Remark}
\newtheorem{remintro}[thmintro]{Remark}

\numberwithin{equation}{section}

\begin{document}
\pagestyle{fancy}
\renewcommand{\sectionmark}[1]{\markright{\thesection\ #1}}
\fancyhead{}
\fancyhead[LO,R]{\bfseries\footnotesize\thepage}
\fancyhead[LE]{\bfseries\footnotesize\rightmark}
\fancyhead[RO]{\bfseries\footnotesize\rightmark}
\chead[]{}
\cfoot[]{}
\setlength{\headheight}{1cm}

\author{Aravind Asok\thanks{Aravind Asok was partially supported by National Science Foundation Awards DMS-1254892 and DMS-1802060.} \;\;\;\;\;\;\; \& \;\;\;\;\;\;\; Jean Fasel \\ \small{\MakeLowercase{with an appendix by} Mrinal Kanti Das}
}
\title{{\bf Euler class groups and motivic stable cohomotopy}}
\date{}
\maketitle

\begin{abstract}
We study maps from a smooth scheme to a motivic sphere in the Morel--Voevodsky $\aone$-homotopy category, i.e., motivic cohomotopy sets.  Following Borsuk, we show that, in the presence of suitable hypotheses on the dimension of the source, motivic cohomotopy sets can be equipped with functorial abelian group structures.

We then explore links between motivic cohomotopy groups, Euler class groups \`a la Nori--Bhatwadekar--Sridharan and Chow--Witt groups. We show that, again under suitable hypotheses on the base field $k$, if $X$ is a smooth affine $k$-variety of dimension $d$, then the Euler class group of codimension $d$ cycles coincides with the codimension $d$ Chow--Witt group; the identification proceeds by comparing both groups with a suitable motivic cohomotopy group.  As a byproduct, we describe the Chow group of zero cycles on a smooth affine $k$-scheme as the quotient of the free abelian group on zero cycles by the subgroup generated by reduced complete intersection ideals; this answers a question of S. Bhatwadekar and R. Sridharan.
\end{abstract}

\begin{footnotesize}
\tableofcontents
\end{footnotesize}

\section*{Introduction}
Suppose $k$ is a field, and $X$ is a smooth affine $k$-scheme.  The goal of this paper is to establish concrete connections between ``obstruction groups" attached to $X$ in the sense of M. Nori--S. Bhatwadekar--R. Sridharan (e.g., Euler class or weak Euler class groups) and ``motivic groups" attached to $X$ (e.g., Chow or Chow--Witt groups).  In brief, we will (i) show that obstruction groups can be interpreted as ``cohomotopy groups" and (ii) using an algebro-geometric analog of the classical comparison between cohomotopy and cohomology groups, identify ``obstruction groups" with motivic groups in a number of cases.

We begin by discussing the historical setting for the problems we consider.  Suppose $X$ has dimension $d$ and $\mathscr{E}$ is a rank $d$ vector bundle on $X$.  By means of the dictionary of Serre \cite{Serre55}, one views this setup as analogous to that of a rank $d$ {\em real} vector bundle over a smooth manifold of dimension $d$.  In the topological setup, there is precisely one cohomological obstruction to writing a given rank $d$ vector bundle as a Whitney sum of a bundle of rank $d-1$ and trivial bundle of rank $1$, i.e., the triviality of the Euler class (taking values in integral cohomology, twisted by an orientation character associated with the bundle) \cite{Milnor74}.

In parallel with the situation in topology, one would like to develop a theory of Euler classes in algebraic geometry allowing one to answer the question: what is the (primary) obstruction to $\mathscr{E}$ splitting as the sum of a rank $d-1$ vector bundle and a trivial bundle of rank $1$?  When $k$ is algebraically closed, this question was completely answered by M.P. Murthy \cite{Murthy94}, building on earlier work \cite{MurthySwan,Mohan82}: under these hypotheses, $\mathscr{E}$ splits if and only if $0 = c_d(\mathscr{E}) \in CH^d(X)$.  If $k$ is not algebraically closed, it has been known for a long time that triviality of the top Chern class does not necessarily guarantee existence of a splitting (e.g., the tangent bundle on the real algebraic sphere of dimension $2$ provides a counterexample).

In response to the splitting problem above, different ``Euler class theories" have been constructed.  Motivated by the work of Murthy and following ideas of M.V. Nori, S. Bhatwadekar and R. Sridharan \cite{Bhatwadekar99, Bhatwadekar00, Bhatwadekar02} developed what we will call an ``algebraic" approach to Euler classes.  They gave an explicit description, using generators and relations, of an Euler class group $\mathrm{E}^d(X)$.  Under suitable assumptions, this group houses an Euler class that provides an obstruction to splitting.  More precisely, if $\mathscr{E}$ is an oriented vector bundle (i.e., an algebraic vector bundle on $X$ equipped with a fixed trivialization of the determinant), then there is an associated class $e(\mathscr{E}) \in \mathrm{E}^d(X)$ whose vanishing is equivalent to $\mathscr{E}$ splitting as a sum of an oriented vector bundle of rank $d-1$ and a trivial bundle of rank $1$.

J. Barge and F. Morel gave a ``cohomological" version of the Euler class theory on smooth affine schemes: they introduced Chow--Witt groups and an Euler class for any oriented vector bundle taking values in the top dimensional Chow--Witt group \cite[\S 2.1]{Barge00}.  Establishing the link with the splitting problem is more difficult: Barge--Morel proved that their Euler class governed the splitting problem for varieties of dimension $\leq 2$ \cite[Th\'eor\`eme 2.3]{Barge00} and conjectured that the same result held for higher dimensional smooth affine varieties \cite[\S 2.1 Conjecture]{Barge00}.  In contrast, the approach of Barge--Morel was much more computationally satisfying: Chow--Witt groups do underlie a cohomology theory on smooth schemes having a number of good formal properties.

Barge and Morel observed that there is a comparison map from the Bhatwadekar--Sridharan Euler class groups to Chow--Witt groups \cite[Remarque 2.4]{Barge00} when both groups are defined.  However, beyond this link, the two theories could not be more different.  The Bhatwadekar--Sridharan approach has the benefit of being very explicit and thus making it relatively easy to see that vanishing of the Euler class does in fact control the splitting problem (and moreover gives a theory for possibly singular affine schemes).  However, the groups $\mathrm{E}^d(X)$ are, practically speaking, very difficult to compute: for example, the groups $\mathrm{E}^d(X)$ do not obviously underlie a cohomology theory.

The Barge--Morel conjecture was eventually resolved by F. Morel \cite[Theorem 1.32]{MField} (see also \cite[Theorem 6]{Fasel09c} for the case of varieties of dimension $3$) by introducing a third ``homotopical" approach to Euler classes (the main result of \cite[Theorem 1]{AFComparison} shows that the homotopical approach coincides with the cohomological approach).  Proving that the various approaches to the theory of the Euler class coincide was an interesting and difficult problem (cf. \cite[Remark 1.33(2)]{MField}).

Our approach to the problem of comparing Euler class groups and Chow--Witt groups involves relating both with a third ``cohomotopy" group.  To explain our techniques, we begin by recalling some classical homotopy theory.  Borsuk showed \cite{Borsuk} that, if $M$ is a manifold of dimension $d \leq 2n-2$, the set of free homotopy classes of maps $[M,S^n]$ admits a (functorial) abelian group structure; this set is called the $n$-th cohomotopy group of $M$.  The Hopf classification theorem \cite{Hopf} states that if $\dim M = n$, then the group $[M,S^n]$ coincides with the cohomology group $H^n(M,\Z)$; the isomorphism is induced by a (dual) Hurewicz map $[M,S^n] \to H^n(M,\Z)$ \cite{SpanierCohomotopy}, which we describe momentarily.

In fact, Borsuk showed that if $X$ is any $(n-1)$-connected space and $M$ has dimension $d \leq 2n-2$ as above, then the set $[M,X]$ admits an abelian group structure, functorially in $X$.  Granting this, observe that the first stage of the Postnikov tower for $S^n$ yields a map of $(n-1)$-connected spaces $S^n \to K(\Z,n)$.  One then shows:
\begin{itemize}[noitemsep,topsep=1pt]
\item the (dual) Hurewicz map is simply the map $[M,S^n] \to [M,K(\Z,n)] = H^n(M,\Z)$ induced by functoriality in Borsuk's construction, and
\item the Hopf classification theorem can be deduced by an elementary obstruction theory argument.
\end{itemize}
Via the Freudenthal suspension theorem, one may view Borsuk's group structure as a ``stable" phenomenon in the sense of stable homotopy theory and thus one may view the cohomotopy groups as part of a cohomology theory.

We analyze similar ideas in the algebro-geometric setting.  The idea of studying algebro-geometric cohomotopy groups goes back (at least) to van der Kallen's group law on orbit sets of unimodular rows \cite{vdKallen83}.  We work in the setting of the Morel--Voevodsky $\aone$-homotopy category \cite{MV} in order to access basic homotopic constructions.  Let $Q_{2n}$ be the even-dimensional smooth affine quadric in ${\mathbb A}^{2n+1}$ given by the equation $\sum_{i=1}^n x_iy_i = z(1-z)$.  We observed in \cite[Theorem 2.2.5]{ADF} that $Q_{2n}$ is a sphere from the standpoint of the Morel-Voevodsky $\aone$-homotopy theory \cite{MV}.  As a consequence, if $X$ is a smooth scheme, by analogy with the ideas of Borsuk, we will call the set $[X,Q_{2n}]_{\aone}$ of morphisms in the $\aone$-homotopy category a {\em motivic cohomotopy set} (see Definition~\ref{defn:motiviccohomotopy}).

Paralleling our discussion of cohomotopy above, our goals in this paper are (i) to equip the set of free $\aone$-homotopy classes of maps $[X,Q_{2n}]_{\aone}$ with a functorial abelian group structure, (ii) to study algebro-geometric analogs of the Hurewicz homomorphism and Hopf classification theorem and (iii) to describe a presentation of $[X,Q_{2n}]_{\aone}$ with explicit generators and relations.  Write $\widetilde{CH}^n(X)$ for the Chow--Witt groups defined by J. Barge--F. Morel \cite{Barge00} and studied in detail in \cite{Fasel08a} (see Section~\ref{ss:motivicspheresprimer} for more precise references).  If $X$ has dimension $d \leq 2n-2$, then write $\mathrm{E}^n(X)$ for the Euler class group of Bhatwadekar--Sridharan \cite{Bhatwadekar02} mentioned above (see Definition~\ref{defn:eulerclassgroup} and Remark~\ref{rem:eulerclassgroups}).  

\begin{thmintro}[See Theorems~\ref{thm:abelian}, \ref{thm:comparison} and \ref{thm:EulertoChow} and Proposition~\ref{prop:surjectivityofsegreclasshomom}]
\label{thmintro:cohomotopy}
Suppose $k$ is a field having characteristic not equal to $2$, $n$ and $d$ are integers, $n \geq 2$, and $X$ is a smooth affine $k$-scheme of dimension $d \leq 2n-2$.  
\begin{enumerate}[noitemsep,topsep=1pt]
\item The set $[X,Q_{2n}]_{\aone}$ has a functorial abelian group structure;
\item there is a functorial ``Hurewicz" homomorphism $[X,Q_{2n}]_{\aone} \to \widetilde{CH}^n(X)$, which is an isomorphism if $d \leq n$; and
\item there is a functorial and surjective ``Segre class" homomorphism $s: \mathrm{E}^n(X) \to [X,Q_{2n}]_{\aone}$.
\item If, furthermore, $k$ is infinite and $d \geq 2$, then the morphism $s$ is an isomorphism.
\end{enumerate}
In particular, under the hypotheses in \textup{Point (4)}, if $X$ is a smooth affine $k$-scheme of dimension $d$, then there is a functorial isomorphism
\[
\mathrm{E}^d(X) \isomto \widetilde{CH}^d(X).
\]
\end{thmintro}

\begin{remintro}
Theorem \ref{thmintro:cohomotopy} essentially completely resolves the comparison problem of Barge--Morel mentioned above; see Remark \ref{rem:assumptions} for more detailed discussion of assumptions on the base field.  We say ``essentially" because Chow--Witt groups as used above are well-behaved for smooth schemes over a field.  Recently, M. Schlichting introduced another homotopical approach to the obstruction problem for commutative Noetherian rings with infinite residue fields \cite[Theorem 6.18]{SchlichtingEuler}.  Schlichting's obstruction group coincides with the Chow--Witt group for smooth schemes over an infinite perfect field \cite[Remark 6.19]{SchlichtingEuler}.  Therefore, it makes sense to ask whether the Bhatwadekar--Sridharan Euler class group can be compared with Schlichting's obstruction group for a commutative Noetherian ring with infinite residue fields.
\end{remintro}

Our construction clarifies the functorial and cohomological properties (e.g., pullbacks, Mayer--Vietoris type sequences, products) of Euler class groups developed in \cite{Mandal10} and \cite{Mandal12}; see Remark~\ref{rem:improvedfunctoriality} for more details.  Finally, our approach explains why Euler class groups, \`a la Bhatwadekar--Sridharan, are {\em not} suitable for the study of splitting problems for bundles of rank below the dimension: in fact, for such vector bundles there is no reason one should even be able to attach an ``Euler class" with values in this group.  In contrast, Morel's homotopic approach does yield a suitable obstruction theory for the splitting problem for bundles of rank below the dimension and this theory has been studied in detail in \cite{Asok12c}:  in analogy with the topological situation there are further obstructions beyond the primary Euler class obstruction.

The above result has a concrete ``classical" consequence that can be deduced from the comparison, which we believe is of independent interest.  Let $Z_0(X)$ be the group of zero cycles on $X$, $CI_0(X) \subset Z_0(X)$ be the subgroup generated by zero-dimensional {\em reduced} complete intersections in $X$, and set $\mathrm{E}_0(X) := Z_0(X)/CI_0(X)$ (see Definition~\ref{defn:weakeulerclassgroup}).  Cycles in the subgroup $CI_0(X)$ are known to be rationally equivalent to zero, and there is an induced (surjective) homomorphism $\mathrm{E}_0(X) \to CH_0(X)$ \cite[Lemma 2.5]{Bhatwadekar99}.

\begin{thmintro}[See Theorem~\ref{thm:weakEulerisChow}]
\label{thmintro:mainchow}
If $k$ is an infinite field having characteristic not equal to $2$ and $X$ is a smooth affine $k$-scheme of dimension $d \geq 2$, then the map $\mathrm{E}_0(X) \to CH_0(X)$ is an isomorphism.
\end{thmintro}

\begin{remintro}
The idea that $CH_0(X)$ should be related to complete intersection subvarieties goes back to the work of M. Pavaman Murthy and R. Swan in the 1970s \cite[Theorem 2]{MurthySwan} (see also \cite{WeibelComplete} and the references therein).  The question of whether the homomorphism in Theorem~\ref{thmintro:mainchow} is an isomorphism was posed explicitly by S. Bhatwadekar--R. Sridharan \cite[Remark 3.13]{Bhatwadekar99} (see also the survey of Murthy \cite[Question 5.3]{MurthySurvey}).  That question was already known to have a positive answer if: (i) $k$ is algebraically closed \cite[Theorem 5.2]{MurthySurvey}, (ii) if $k$ is the field of real numbers \cite[Theorem 5.5]{Bhatwadekar99}, or (iii) if $\dim X \leq 2$ (unpublished work of Bhatwadekar).
\end{remintro}

\subsubsection*{Overview}
Section~\ref{s:cohomotopy} is devoted to equipping the set $[X,Q_{2n}]_{\aone}$ with a functorial abelian group structure and studying the properties of this group in detail.  The existence of a functorial abelian group structure on this set of homotopy classes of maps is largely formal--at least given $\aone$-analogs of classical connectivity results.  The necessary connectivity results follow from F. Morel's unstable $\aone$-connectivity theorem and his $\aone$-analog of Freudenthal's classical suspension theorem \cite[\S 6]{MField}.

We give two equivalent constructions of the abelian group structure; the first, using the aforementioned suspension theorem, is useful for analyzing various functorial properties of the group structure, and the second, modeled on Borsuk's original construction, is useful for giving a geometric interpretation of the composition.  In particular, we observe that the abelian group structure on $[X,Q_{2n}]_{\aone}$ arises via an identification with stable cohomotopy groups.  Since stable cohomotopy is a ring cohomology theory, we obtain Mayer--Vietoris-type exact sequences and product structures on these groups.

As mentioned above, the variety $Q_{2n}$ is a motivic sphere, and we appeal to the computations of $\aone$-homotopy sheaves of spheres due to F. Morel \cite[Chapter 5]{MField}.  In particular, one knows that the first non-vanishing $\aone$-homotopy sheaf of $Q_{2n}$ is the $n$-th, and this sheaf is the unramified Milnor--Witt K-theory sheaf \cite[Chapter 3]{MField} (see Proposition~\ref{prop:propertiesofspheres}).  If $K(\K^{MW}_n,n)$ is the Eilenberg--Mac Lane space representing sheaf cohomology \cite[Chapter 8.3 p. 212]{JardineLHT}, then the first non-trivial stage of the $\aone$-Postnikov tower as described in \cite[\S 6.1]{Asok12a} yields a morphism $Q_{2n} \to K(\K^{MW}_n,n)$.  One may define the $n$-th Chow--Witt group of a smooth scheme $X$ as $H^i(X,\K^{MW}_n)$ and the functoriality of our abelian group structure immediately yields the Hurewicz homomorphism in the statement.  Techniques of obstruction theory as studied, e.g., in \cite{Asok12a} or \cite{Asok12c} imply the isomorphism statement.

In Section \ref{s:geometricgroupstructure} we provide a concrete description of the abelian group structure on $[X,Q_{2n}]_{\aone}$; this uses two tools.  First, we appeal to the affine representability results of \cite{AsokHoyoisWendtII}; these results allow us to identify the abstract set $[X,Q_{2n}]_{\aone}$ in terms of ``naive" $\aone$-homotopy classes of morphisms of $k$-schemes $X \to Q_{2n}$.  Then, we use the relationship between naive $\aone$-homotopy classes of morphisms with target $Q_{2n}$ and complete intersection ideals studied in \cite{Fasel15b}.  In short, morphisms $f: \Spec R \to Q_{2n}$ determine (non-uniquely) pairs $(I,\omega_I)$, where $I \subset R$ is an ideal, and $\omega_I: R/I^{\oplus n} \to I/I^2$ is a surjection; the naive $\aone$-homotopy class of $f$ essentially only depends on $I$.  Combining these ideas, we give an explicit ``ideal-theoretic" description of the product in Theorem~\ref{thm:concretetau}.

Finally, Section \ref{s:Euler} studies applications of the above ideas.  Using our concrete description of the product in cohomotopy, the existence of the ``Segre class" homomorphism is straightforward and described in Section~\ref{s:Euler} after some recollections on Euler class groups (we slightly recast the definition to fit more naturally in our context).  Surjectivity in the case of interest follows from a moving lemma and holds with minimal hypotheses.  In order to establish injectivity of the Segre class homomorphism, we define an explicit inverse.  In order to define the inverse homomorphism, we appeal to homotopy invariance of Euler class groups.  Finally, Theorem~\ref{thmintro:mainchow} is deduced from the comparison of Euler class groups and Chow--Witt groups, which appeal to results of Das--Zinna \cite{Das15}.  Appendix \ref{s:properties} provides streamlined treatment of homotopy invariance for Euler class groups (see Theorem~\ref{homotopy}) and some variations on the techniques of \cite{Das15} (see Proposition~\ref{prop:dzhomomrevisited}) tailored for this paper.  With these results in mind, Remark~\ref{rem:assumptions} explains the presence of the hypotheses in the statements above.  

\subsubsection*{Note on related work}
The first version of this paper was posted to the ArXiv in January 2016.  The original argument for injectivity of the map $s$ in Theorem~\ref{thmintro:cohomotopy}(4) contained a gap; this version closes that gap.  In October 2016, S. Mandal and B. Mishra posted a preliminary version of \cite{MandalMishra}, the results of which have some overlap with this one.  Using our notation, and under suitable hypotheses on $R$, they equip $\pi_0(\Singaone Q_{2n}(R))$ with an explicit group structure \cite[Theorem 6.5]{MandalMishra} defined ideal-theoretically (cf. Theorem~\ref{thm:concretetau} below); comparing \cite[Definition 6.1]{MandalMishra} and Construction~\ref{construction:geometriccomposition} below, one deduces that their group structure coincides with all the variants studied in this paper.  Again under suitable hypotheses on $R$, Mandal and Mishra construct \cite[\S 7]{MandalMishra} a surjective map $\mathrm{E}_n(R) \to \pi_0(\Singaone Q_{2n}(R))$ (cf. Proposition~\ref{prop:surjectivityofsegreclasshomom} below).  Furthermore, they prove (see \cite[Theorem 7.3]{MandalMishra}) that the preceding map is an isomorphism (cf. Theorem~\ref{thm:comparison} below).  Nevertheless, the techniques of \cite{MandalMishra} differ rather significantly from those used in this paper.

\subsubsection*{Notation/Preliminaries}
In this paper, the word ring will mean always mean commutative unital ring.  If $R$ is a ring and $a=(a_1,\ldots,a_m)\in R^m$, we write $\langle a\rangle\subset R$ for the ideal generated by $a_1,\ldots,a_m$.  If $I$ is an ideal, we write $\mathrm{ht}(I)$ for the height of $I$.

Fix a base field $k$ and write $\Sm_k$ for the subcategory of schemes over $\Spec k$ that are separated, smooth and have finite type over $\Spec k$.  We consider the categories $\Spc_k$ and $\Spc_{k,\bullet}$ of simplicial and pointed simplicial presheaves on $\Sm_k$; objects of these categories will be called {\em (pointed) $k$-spaces}; if $k$ is clear from context, we will simply call objects of this category (pointed) spaces.

We equip the category of (pointed) simplicial presheaves on $\Sm_k$ with its usual injective Nisnevich local model structure \cite[\S 5.1]{JardineLHT}; the associated homotopy category, which we denote by $\mathscr{H}^{\Nis}_s(k)$ ($\mathscr{H}^{\Nis}_{s,\bullet}(k)$), will be referred to as the {\em (pointed) simplicial homotopy category}.  If $(\mathscr{X},x)$ and $(\mathscr{Y},y)$ are pointed spaces, we set $[\mathscr{X},\mathscr{Y}]_s := \hom_{\mathscr{H}^{\Nis}_s(k)}(\mathscr{X},\mathscr{Y})$ and $[(\mathscr{X},x),(\mathscr{Y},y)]_s := \hom_{\mathscr{H}^{\Nis}_{s,\bullet}(k)}(\mathscr{X},\mathscr{Y})$.  An element of $[\mathscr{X},\mathscr{Y}]_s$ will also be called a free simplicial homotopy class of maps.

The category of (pointed) simplicial presheaves can be further localized to obtain the Morel-Voevodsky $\aone$-homotopy category $\ho{k}$ ($\hop{k}$); this localization is a left Bousfield localization of $\Spc_k$ \cite{MV}.  In particular, there is an endo-functor $\Laone$ of the category of (pointed) simplicial presheaves, together with a natural transformation $\theta: \id \to \Laone$ such that if $\mathscr{Y}$ is a space, then $\mathscr{Y} \to \Laone \mathscr{Y}$ is a cofibration and $\aone$-weak equivalence and $\Laone \mathscr{Y}$ is simplicially fibrant and $\aone$-local.  We refer the reader to \cite[Proposition 2.2.1]{AsokWickelgrenWilliams} for a convenient summary of properties of the $\aone$-localization functor and note in passing that $\Laone$ commutes with the formation of finite products.  We set $[\mathscr{X},\mathscr{Y}]_{\aone} := \hom_{\ho{k}}(\mathscr{X},\mathscr{Y})$ and $[(\mathscr{X},x),(\mathscr{Y},y)]_{\aone} := \hom_{\hop{k}}(\mathscr{X},\mathscr{Y})$.  Similarly, an element of $[\mathscr{X},\mathscr{Y}]_{\aone}$ will be called a free $\aone$-homotopy class of maps.

We begin by recalling some notation.  Write $S^1$ for the simplicial circle and $S^i$ for the simplicial $i$-sphere, i.e., the $i$-fold smash product of $S^1$ with itself.  More generally, set $S^{i+j,j} := S^i \sma \gm{\sma j}$.  If $\mathscr{Y}$ is any pointed space, then we write $\Sigma^i \mathscr{Y}$ for the smash product $S^i \sma \mathscr{Y}$ and $\Omega^i \mathscr{Y}$ for the derived $i$-fold loop space i.e., $\hom_{\bullet}(S^i,\mathscr{Y}^f)$, where $(-)^f$ is a functorial fibrant replacement functor.  Looping and suspension are adjoint, and the unit map of the loop-suspension adjunction yields a functorial morphism $\mathscr{Y} \to \Omega^i \Sigma^i \mathscr{Y}$ for any integer $i$.

A pointed space $(\mathscr{X},x)$ will be called simplicially $m$-connected if its stalks are all $m$-connected simplicial sets.  Similarly, we will say that $\mathscr{X}$ is $\aone$-$m$-connected, if $\Laone \mathscr{X}$ is simplicially $m$-connected.  We freely use K. Brown's homotopy-theoretic formulation of sheaf cohomology; we refer the reader to \cite[Part III]{JardineLHT} for more details.  For example, if $\bpi$ is a Nisnevich sheaf of abelian groups on $\Sm_k$ there are Eilenberg--Mac Lane objects $K(\bpi,n)$ \cite[p. 212]{JardineLHT} such that $[X,K(\bpi,n)]_s \cong H^n_{\Nis}(X,\bpi)$ \cite[Theorem 8.25]{JardineLHT}.  Unless we mention otherwise, sheaf cohomology will always be taken with respect to the Nisnevich topology.

\subsubsection*{Acknowledgments}
The authors would like to thank Satya Mandal for helpful conversations and correspondence and Chuck Weibel and Pavaman Murthy for helpful correspondence. We also would like to thank Yong Yang for helpful conversations. Moreover, we thank the participants of the AIM workshop on ``Projective modules and $\aone$-homotopy theory" for providing a very stimulating environment. Theorem~\ref{thmintro:cohomotopy} answers Question 2 from \url{http://aimath.org/pastworkshops/projectiveA1problems.pdf}, which was one of the main motivating questions at the workshop.  Finally, we thank the referees of previous versions of this paper for numerous suggestions to improve correctness and clarity of the presentation.

\section{Motivic stable cohomotopy}
\label{s:cohomotopy}
In this section, we establish an analog in the $\aone$-homotopy category of a result of Borsuk: if $U \in \Sm_k$ has Krull dimension $\leq 2n-2$ and $(\mathscr{X},x)$ is a pointed $\aone$-$(n-1)$-connected space, then, if $n \geq 2$, the set $[U,\mathscr{X}]_{\aone}$ of free $\aone$-homotopy classes of maps admits an abelian group structure, functorial in both inputs.  Section~\ref{ss:abgroupmapping} studies the properties of the construction in this generality.  In Section~\ref{ss:cohomotopy} we specialize these results to the cases of interest.  The main results used from this section in subsequent sections are Theorems~\ref{thm:naivevsrefined} and \ref{thm:abelian}.

\subsection{A primer on motivic spheres}
\label{ss:motivicspheresprimer}
Let $Q_{2n-1}$ be the smooth quadric in ${\mathbb A}^{2n}_{\Z}$ defined by the equation $\sum_{i=1}^n x_iy_i = 1$.  Let $Q_{2n}$ be the smooth quadric in ${\mathbb A}^{2n+1}_{\Z}$ defined by the equation $\sum_{i=1}^n x_i y_i = z(1-z)$.  In \cite{ADF}, the notation $Q_{2n}$ was used for an isomorphic quadric, namely change variables $x_i \mapsto -x_i$, and $z \mapsto -z$ for the isomorphism, and we implicitly use this isomorphism to appeal to the results of \cite{ADF} below.  We equip $Q_{2n}$ with the base-point given by the class of $x_1 = \cdots = x_n = y_1 = \cdots = y_n = 0,z = 1$.  Our goal in this section is to recall various facts about the geometry and $\aone$-homotopy theory of such quadrics.

\subsubsection*{Naive homotopy classes}
If $R$ is a ring, write as usual $\Delta^{\bullet}_R$ for the cosimplicial affine space over $R$, i.e.,
\[
\Delta^n_R := \Spec R[t_0,\ldots,t_n]/(\sum_i t_i = 1),
\]
equipped with the usual coface and codegeneracy morphisms.  If $Y$ is a (pointed) smooth $k$-scheme, then the assignment sending any $k$-algebra $R$ to
\[
\Singaone Y(R) := Y(\Delta^{\bullet}_R)
\]
defines a presheaf of simplicial sets on the category of smooth affine schemes.  We refer the reader to \cite[p. 87]{MV} for more details; this construction extends in an evident way to a presheaf of simplicial sets on $\Sm_k$.

Since $\Singaone Y$ is a presheaf of simplicial sets, we may take connected components to obtain the naive connected components presheaf
\[
\pi_0(\Singaone Y)(R) := \pi_0(\Singaone Y(R)).
\]
The presheaf $\pi_0(\Singaone Y)$ extends in an evident way to a presheaf of sets on $\Sm_k$ that is pointed if $Y$ is pointed.  The set $\pi_0(\Singaone Y)(R)$ is also called the set of naive $\aone$-homotopy classes of maps $\Spec R \to Y$.  The following result connects the set of naive homotopy classes of maps to ``true" $\aone$-homotopy classes (see \cite[Definition 2.1.1]{AsokHoyoisWendtII} for the definition of $\aone$-naive simplicial presheaves).

\begin{thm}
\label{thm:naivevsrefined}
Assume $k$ is a field, $X = \Spec R$ is a smooth affine $k$-scheme and $n \geq 0$ is an integer; if $n$ is even and $\geq 8$ assume further that $k$ has characteristic different from $2$.  The quadric $Q_n$ is $\aone$-naive; in particular, the map
\[
\pi_0(\Singaone Q_{n}(R)) \longrightarrow [X,Q_{n}]_{\aone}
\]
is a bijection, functorially in $X$.
\end{thm}

\begin{proof}
See {\cite[Theorems 4.2.1 and 4.2.2.]{AsokHoyoisWendtII}} for the case where either $n$ is odd or $n$ is even and $\leq 4$.  If $n = 6$, then combine \cite[Theorem 2.3.5 and Proposition 3.1.1]{AHWOctonion} with \cite[Theorem 2.2.4]{AsokHoyoisWendtII}.  Finally to treat the case where $k$ is finite, we appeal to \cite[Theorem 2.11]{AsokHoyoisWendtIII}.
\end{proof}

\begin{rem}
Below, we will routinely use the following consequence of this result: every element of $f \in [X,Q_{n}]_{\aone}$ can be represented by an actual morphism of schemes $f: X \to Q_{n}$.  In particular, building off this result and Proposition~\ref{prop:comparingcompositions}, we will give a geometric description of the abelian group structure on $[X,Q_{n}]_{\aone}$ from Proposition~\ref{prop:groupstructureI}.
\end{rem}

\subsubsection*{Base change and perfect subfields}
We refer the reader to \cite[Chapter 8.3 p. 212]{JardineLHT} for a discussion of sheaf cohomology formulated in terms of the local homotopy theory and \cite[\S 3.3]{AFComparison} for some complementary results. Recall that a Nisnevich sheaf of groups $\mathbf{G}$ is called {\em strongly $\aone$-invariant} if the simplicial classifying space $B\mathbf{G}$ is $\aone$-local.  A Nisnevich sheaf of abelian groups $\mathbf{A}$ is called {\em strictly $\aone$-invariant} if the cohomology presheaves $U\mapsto H^i(U,\mathbf{A})$ are homotopy invariant for any $i\geq 0$.  A sheaf of abelian groups $\mathbf{A}$ is strictly $\aone$-invariant if and only if the Eilenberg-Mac Lane spaces $K(\mathbf{A},i)$ are $\aone$-local for every $i \geq 0$.  We use the following variant of Morel's unstable $\aone$-connectivity property (see \cite[\S 2.2-2.3]{AsokWickelgrenWilliams} for an axiomatic treatment).

\begin{thm}[Morel]
\label{thm:unstableconnectivity}
Assume $k$ is a field, and $F \subset k$ is a perfect subfield.  Suppose $(\mathscr{X},x) \in \Spc_{k,\bullet}$ is a pointed space that is pulled back from $F$.  The following statements hold:
\begin{enumerate}[noitemsep,topsep=1pt]
\item For any integer $i \geq 1$ (resp. $i \geq 2$), the homotopy sheaves $\bpi_i^{\aone}(\mathscr{X})$ are strongly $\aone$-invariant (resp. strictly $\aone$-invariant).
\item If furthermore, $\bpi_1^{\aone}(\mathscr{X})$ is abelian, then it is strictly $\aone$-invariant.
\end{enumerate}
\end{thm}

\begin{proof}
The results hold for $F$ by appeal to \cite[Theorem 6.1]{MField} and \cite[Theorem 5.46]{MField}.  The results for $k$ follow immediately by standard base-change results \cite[Lemmas A.2 and A.4]{HoyoisHM}.
\end{proof}

\subsubsection*{Connectivity and cohomology of quadrics}
We now introduce a notion of ``cohomological dimension" for a space by restricting attention to strictly $\aone$-invariant sheaves.

\begin{defn}
If $X \in \Sm_k$, then say that $X$ has {\em $\aone$-cohomological dimension $\leq d$} if, for any strictly $\aone$-invariant sheaf $\mathbf{B}$, $H^i(X,\mathbf{B})$ vanishes for $i > d$.  Likewise, say $X$ has $\aone$-cohomological dimension $d$ (write $d = cd_{\aone}(X)$) if $d$ is the smallest integer such that $X$ has $\aone$-cohomological dimension $\leq d$.
\end{defn}

We now collect some results describing the sense in which the varieties $Q_i$ are motivic spheres and the sense in which they behave, cohomologically, like spheres in classical homotopy theory.

\begin{prop}
\label{prop:propertiesofspheres}
Assume $k$ is a field.  The following statements hold:
\begin{enumerate}[noitemsep,topsep=1pt]
\item the quadric $Q_i$ is a smooth affine scheme of dimension $i$;
\item $Q_{2n-1}$ is $\aone$-weakly equivalent to $\Sigma^{n-1}\gm{\sma n}$ and $Q_{2n}$ is $\aone$-weakly equivalent to $\Sigma^n \gm{\sma n}$; thus
    \begin{enumerate}[noitemsep,topsep=1pt]
    \item the quadric $Q_{i}$ is at least $\aone$-$(\lfloor \frac{i}{2} \rfloor - 1)$-connected;
    \item the first non-vanishing homotopy sheaf of $Q_i$ appears in degree $\lfloor \frac{i}{2} \rfloor$ and, if $i = 2n-1$ or $i = 2n$, is isomorphic to $\K^{MW}_n$;
    \end{enumerate}
\item if $n > 0$, for any strictly $\aone$-invariant sheaf $\mathbf{M}$,
\[
\begin{split}
H^j(Q_{2n},\mathbf{M}) = \begin{cases} \mathbf{M}(k) & \text{ if } j = 0; \\ \mathbf{M}_{-n}(k) & \text{ if } j = n; \\ 0 & \text{ otherwise;}\end{cases}\;\;\;\;
H^j(Q_{2n+1},\mathbf{M}) = \begin{cases} \mathbf{M}(k) & \text{ if } j = 0; \\ \mathbf{M}_{-n-1}(k) & \text{ if } j = n; \\ 0 & \text{ otherwise.}\end{cases}
\end{split}
\]
\end{enumerate}
\end{prop}

\begin{proof}
The first statement is immediate from the definition.  The second statement is \cite[Theorem 2.2.5]{ADF}; the substatements follow from Morel's unstable $\aone$-connectivity theorem \cite[Theorem 6.38]{MField} since an $i$-fold simplicial suspension is at least simplicially $(i-1)$-connected and Morel's computation of the first non-vanishing $\aone$-homotopy sheaf of a motivic sphere \cite[Theorem 6.40]{MField}; we implicitly appeal to Theorem \ref{thm:unstableconnectivity} to make statements over an arbitrary field (since spheres are defined over $\Spec \Z$).   Since the map $Q_{2n+1} \to {\mathbb A}^{n+1} \setminus 0$ (say projecting onto the $x$-variables) is an $\aone$-weak equivalence, the final statement follows from \cite[Lemma 4.5]{Asok12b} and the suspension isomorphism for cohomology.
\end{proof}

\subsubsection*{Chow--Witt groups of spheres}
Chow--Witt groups were initially defined in \cite{Barge00} (though they were called ``oriented Chow groups" there) and studied in detail in \cite{Fasel07,Fasel08a}.  The original definition is via cohomology of an explicit ``twisted Gersten complex" (see, e.g., \cite[D\'efinition 10.2.16]{Fasel08a}).  In \cite[Theorem 2.11 (p. 611)]{AFComparison}, we showed that for any line bundle $\mathscr{L}$ on a smooth scheme $X$, there is a twisted version of the Milnor--Witt K-theory sheaf $\K^{MW}(\mathscr{L})$ on $X$ and an identification
\[
H^n(X,\K^{MW}_n(\mathscr{L})) \cong \widetilde{CH}^n(X,\mathscr{L})
\]
that respects the contravariant functoriality on $\Sm_k$ of each side.  (Note: a priori the twisted Chow--Witt groups on the right hand side are Zariski cohomology groups, among other things, Zariski and Nisnevich cohomology groups of (twisted) strictly $\aone$-invariant sheaves on smooth schemes coincide).

Chow--Witt groups also are equipped with localization sequences: if $Z \subset X$ is a closed immersion of smooth schemes of codimension $c$ and normal bundle $\nu$, then there is a localization long exact sequence
\[
\cdots \widetilde{CH}^{i-c}(Z,\det \nu^{\vee}) \longrightarrow \widetilde{CH}^i(X) \longrightarrow \widetilde{CH}^i(X \setminus Z) \longrightarrow \cdots
\]
where $\det \nu$ is the determinant of $\nu$ and the right hand map is not in general surjective \cite[Corollaire 10.4.11]{Fasel08a}.

The subvariety $Z_n \subset Q_{2n}$ defined by $x_1 = \cdots = x_n = z = 0$ is isomorphic to ${\mathbb A}^n$ and comes equipped with a trivialization of the normal bundle.  The complement $X_{2n} := Q_{2n} \setminus Z_n$ is $\aone$-contractible by \cite[Theorem 3.1.1]{ADF}.  Using the localization sequence for this inclusion and $\aone$-homotopy invariance of Chow--Witt groups \cite[Corollaire 11.3.3]{Fasel08a}, there is a distinguished isomorphism $H^n(Q_{2n},\K^{MW}_n) \cong \K^{MW}_0(k)$.

We want to write an explicit generator for $H^n(Q_{2n},\K^{MW}_n)$; we do this in terms of the Gersten--Schmid resolution of $\K^{MW}_n$ \cite[\S 5.1 and Corollaries 5.43-5.44]{MField}.  The generic point of $Z_n$ is an element $z_n$ of $Q_{2n}^{(n)}$.  The maximal ideal of $\mathscr{O}_{Q_{2n},z_n}$ is generated by the classes of $x_1,\ldots,x_n$ and we can consider $x_1\wedge\ldots\wedge x_n\in \wedge^n(\mathfrak m_{z_n}/\mathfrak m_{z_n}^2)$. Finally, we obtain an element $\alpha = \langle 1\rangle\otimes x_1\wedge\ldots\wedge x_n$ in $\mathbf{K}^{MW}_{0}(k(z_n),\wedge^n (\mathfrak m_{z_n}/\mathfrak m_{z_n}^2))$; by unwinding the devissage isomorphism, one may establish the following result, which shows that the element just described provides an explicit generator for $\CH^n(Q_{2n})$.

\begin{lem}[{\cite[Lemma 4.2.6]{ADF}}]
\label{lem:explicitgenerator}
The element $\alpha_n\in \mathbf{K}^{MW}_{0}(k(z_n),\wedge^n (\mathfrak m_{z_n}/\mathfrak m_{z_n}^2))$ is a cycle in the Gersten--Schmid complex and its class generates $\CH^n(Q_{2n})$ as a $\mathbf{K}_0^{MW}(k)$-module.
\end{lem}

\subsection{Abelian group structures on mapping sets}
\label{ss:abgroupmapping}
In this section, we give two equivalent constructions of a functorial abelian group structure on sets of homotopy classes of maps.  The second is essentially Borsuk's classical construction, transplanted in our context; this version has the benefit that it can be made very explicit and will be used to derive concrete formulas for the composition in special cases (see Proposition~\ref{prop:groupstructureII}).  The first construction we present is a modernized version of Borsuk's construction (see Proposition~\ref{prop:groupstructureI}); this version has the benefit of rendering various properties of the group structure entirely formal.  The equivalence of the two constructions is established in Proposition~\ref{prop:comparingcompositions} and various functorial properties, including the existence of the Hurewicz homomorphism, are studied in Proposition~\ref{prop:functorialityinthetarget}.

\subsubsection*{Group structures on mapping sets: a modern approach}
Recall the notion of $h$-cogroup \cite[Definition 2.2.7]{Arkowitz} and $h$-group \cite[Definition 2.2.1]{Arkowitz}.  The quintessential example of an $h$-cogroup is given by the classical sphere $S^i_{top}$, $i \geq 1$, and the fold map $S^i_{top} \to S^i_{top} \vee S^i_{top}$.  By choosing a suitable simplicial model of $S^i_{top}$, one obtains an $h$-cogroup structure on the simplicial sphere $S^i$.  It follows that the constant simplicial presheaf $S^i$ has the structure of an $h$-co-group object in the category of simplicial presheaves, inherited from the sectionwise $h$-cogroup structure of the simplicial set $S^i$.  More generally, one obtains a functorial in $\mathscr{Y}$ $h$-cogroup structure on $\Sigma^i \mathscr{Y}$ for any pointed space $\mathscr{Y}$.

By adjunction, and appeal to the corresponding facts for simplicial sets, the space $\Omega^i \mathscr{Y}$ then carries a functorial $h$-group structure; if $i \geq 2$ it has a functorial structure of homotopy commutative $h$-group.  In particular, for any $U \in \Sm_S$, the set of pointed maps $[U_+,\Omega^i \Sigma^i \mathscr{Y}]$ has the structure of a group, functorially in both inputs; this group is automatically abelian whenever $i \geq 1$.

By appeal to the loop-suspension adjunction, for any pointed space $\mathscr{Y}$, there are maps $\mathscr{Y} \to \Omega^i \Sigma^i \mathscr{Y}$ for every $i \geq 0$.  The Freudenthal suspension theorem identifies the connectivity of the map $\mathscr{Y} \to \Omega^i \Sigma^i \mathscr{Y}$ at least under assumptions on the connectivity of $\mathscr{Y}$.   By appeal to obstruction theory arguments, we will now transport the group structure on $[U_+,\Omega^i \Sigma^i \mathscr{Y}]$ to one on $[U,\mathscr{Y}]_{\aone}$, at least under suitable hypotheses on the Krull dimension of $U$ and on the connectivity of $\mathscr{Y}$.

\begin{prop}
\label{prop:groupstructureI}
Assume $k$ is a field, and $n \geq 2$ is an integer.  Suppose $\mathscr{X}$ is a pointed $k$-space that is $\aone$-$(n-1)$-connected and pulled back from a perfect subfield of $k$ (e.g., the prime field).  If $U \in \Sm_k$ has $\aone$-cohomological dimension $d \leq 2n-2$, then for any integer $i \geq 1$, the map
\[
[U,\mathscr{X}]_{\aone} \cong [U_+,(\mathscr{X},x)]_{\aone} \longrightarrow [U_+,\Omega^i \Sigma^i (\mathscr{X},x)]_{\aone}
\]
induced by the pointed map $(\mathscr{X},x)\to \Omega^i \Sigma^i (\mathscr{X},x)$ is a bijection, functorial in both inputs.  Moreover, the group structure on $[U,\mathscr{X}]_{\aone}$ induced in this fashion is abelian, and functorial with respect to both inputs.
\end{prop}

\begin{proof}
Under the stated assumptions on $k$, we may appeal to the results of Morel in \cite[\S 6]{MField}; see \cite[\S 2-3]{AsokWickelgrenWilliams} for a more axiomatic treatment of these results.  We begin by collecting a number of connectivity statements.  By assumption, $\mathscr{X}$ is at least $\aone$-$(n-1)$-connected for some integer $n \geq 2$.  In particular, it is at least $\aone$-$1$-connected.  Note that $\Laone \mathscr{X}$ is then simplicially $(n-1)$-connected by definition, and thus $\Sigma^i \Laone \mathscr{X}$ is at least simplicially $(n+i-1)$-connected.  Morel's unstable $\aone$-connectivity theorem \cite[Theorem 6.38]{MField} implies that $\Sigma^i \Laone \mathscr{X}$ is also at least $\aone$-$(n+i-1)$-connected.  Since the map $\mathscr{X} \to \Laone \mathscr{X}$ is an $\aone$-weak equivalence, we conclude by appeal to \cite[\S 2 Lemma 2.13]{MV} that $\Sigma^i \mathscr{X}$ is also at least $\aone$-$(n+i-1)$-connected.

Since $n \geq 2$, we deduce that the canonical map $\Laone \Omega \Sigma \mathscr{X} \to \Omega \Laone \Sigma \mathscr{X}$ is a simplicial weak equivalence by \cite[Theorem 6.46]{MField} (see also \cite[Theorem 2.4.1]{AsokWickelgrenWilliams}).  For later use, the above connectivity estimates and an induction argument guarantee that the map $\Laone \Omega^i \Sigma^i \mathscr{X} \to \Omega^i \Laone \Sigma^i \mathscr{X}$ is a simplicial weak equivalence as well.  Since $\mathscr{X}$ is at least $\aone$-$(n-1)$-connected, the above connectivity estimates also allow us to conclude that $\Omega^i \Laone \Sigma^i \mathscr{X}$ is at least $\aone$-$(n-1)$-connected.

Morel's $\aone$-suspension theorem \cite[Theorem 6.61]{MField} states that, under the above hypotheses, the map $\mathscr{X} \to \Omega \Laone \Sigma \mathscr{X}$ is an $\aone$-$(2n-2)$-connected (i.e., the $\aone$-homotopy fiber $\mathscr{F}$ of this morphism is at least $\aone$-$(2n-2)$-connected).  Granted this, we may appeal to the $\aone$-Moore--Postnikov factorization of the map $\mathscr{X} \to \Omega \Laone \Sigma \mathscr{X}$ \cite[Theorem 6.1.1]{Asok12c} in order to analyze whether a pointed map $f: U_+ \to \Omega \Laone \Sigma \mathscr{X}$ lifts to $\mathscr{X}$.  Since the fundamental groups of $\mathscr{X} \to \Omega \Laone \Sigma \mathscr{X}$ are both trivial, the $\aone$-Moore--Postnikov factorization becomes a tower of principal fibrations (so the statement of \cite[Theorem 6.1.1]{Asok12c} is considerably simplified).

Now, suppose we have a morphism $f: U_+ \to \Omega \Laone \Sigma \mathscr{X}$.  Using the assumption on the $\aone$-cohomological dimension of $U$ and appealing to the lifting procedure described on \cite[p. 1055]{Asok12c}, one sees that $f$ lifts uniquely up to $\aone$-homotopy to a map $\tilde{f}: U_+ \to \mathscr{X}$.  In more detail, the obstructions to lifting $U_+$ to a morphism $U_+ \to \mathscr{X}$ are inductively defined elements of the group $H^{i+1}(U,\bpi_i(\mathscr{F}))$.  If a given obstruction vanishes, then the space of lifts is parameterized by a quotient of $H^i(U,\bpi_i(\mathscr{F}))$.  Since $\bpi_i(\mathscr{F}) = 0$ for $i \leq 2n-2$ by assumption, and $H^i(U,\mathbf{M}) = 0$ for any strictly $\aone$-invariant sheaf $\mathbf{M}$ and any integer $i > 2n-2$ by the assumption on the $\aone$-cohomological dimension of $U$, it follows that the obstruction groups at every stage are always trivial, and that there is a uniquely defined lift at every stage.  Furthermore, by \cite[Lemma 2.1]{Asok12b}, since $\mathscr{X}$ is $\aone$-$1$-connected, the map $[U_+,\mathscr{X}]_{\aone} \to [U,\mathscr{X}]_{\aone}$ ``forgetting the base-point" and the corresponding map with $\mathscr{X}$ replaced by $\Omega \Sigma \mathscr{X}$ are bijections, functorially in both inputs.  By transport of structure, $[U,\mathscr{X}]_{\aone}$ inherits a group structure.

To see that the group structure on $[U,\mathscr{X}]_{\aone}$ described above is abelian is also straightforward.  The connectivity estimates of the first paragraph combined with Morel's simplicial suspension theorem imply that the map $\Sigma \mathscr{X} \to \Omega \Sigma^2 \mathscr{X}$ is $\aone$-$2n$-connected and the map $\Omega \Sigma \mathscr{X} \to \Omega^2 \Sigma^2 \mathscr{X}$ is $\aone$-$(2n-1)$-connected.  Since the map in the previous line is an $h$-map and $\Omega^2 \Sigma^2 \mathscr{X}$ is homotopy commutative, for $U$ as in the statement the induced map $[U,\Omega \Sigma \mathscr{X}]_{\aone} \to [U,\Omega^2 \Sigma^2 \mathscr{X}]_{\aone}$ is an isomorphism of groups and thus the former is necessarily abelian; this also establishes the result for $i = 1$, and also for $i = 2$.  For $i \geq 2$, one proceeds inductively and shows that the map $[U,\Omega^{i-1} \Sigma^{i-1} \mathscr{X}]_{\aone} \to [U,\Omega^i\Sigma^i \mathscr{X}]_{\aone}$ is always a bijection.
\end{proof}

We now show that set of free homotopy classes of maps $[U,\mathscr{X}]_{\aone}$ can be identified with maps in a stable homotopy category.  Write $\mathrm{SH}^{S^1}_{\aone}$ for the $S^1$-stable $\aone$-homotopy category (see, for example, \cite[Definition 4.1.1]{MStable}).  This category may be constructed using an ``injective" model structure similar to that we used for the unstable $\aone$-homotopy category.  However, filtered colimits of fibrant objects are not particularly well-behaved in the injective model structure, so it will be helpful to use different model structures, namely the motivic model structure of \cite[Theorem 2.12]{DundasRondigsOstvaer}; in this model structure filtered colimits in $\Spc_{k,\bullet}$ preserve fibrant objects by \cite[Corollary 2.16]{DundasRondigsOstvaer}.

Recall from \cite[Definition 2.3]{DundasRondigsOstvaer} that a (pointed) simplicial presheaf $\mathscr{F}$ is fibrant with respect to the motivic model structure if for every smooth scheme $X$ the following conditions hold: (i)$\mathscr{F}(X)$ is a Kan complex, (ii) the projection $X \times \aone \to X$ induces a weak equivalence of simplicial sets $\mathscr{F}(X) \to \mathscr{F}(X \times \aone)$, (iii) $\mathscr{F}$ satisfies Nisnevich excision, i.e., $\mathscr{F}$ takes Nisnevich distinguished squares to homotopy pullback squares and $\mathscr{F}(\emptyset)$ is contractible.  The category $\Spc_{k,\bullet}$ is a simplicial model category with the usual notion of simplicial mapping space $\map(\mathscr{X},\mathscr{Y})$ \cite[p. 47]{MV}.  The weak equivalences in the motivic model structure on $\Spc_{k,\bullet}$ coincide with the Morel--Voevodsky weak equivalences, and \cite[Theorem 2.17]{DundasRondigsOstvaer} shows that the identity functor induces a Quillen equivalence between this model category and the version of the unstable $\aone$-homotopy category we have used so far.

Write $\Spt_{k}$ for the category of $S^1$-spectra of motivic spaces: this is constructed just as in \cite[\S 2.2]{DundasRondigsOstvaer}, i.e., an $S^1$-spectrum is a sequence of motivic spaces $E_n$ together with structure maps $\sigma: \Sigma E_n \to E_{n+1}$, and morphisms of spectra are morphisms of sequences commuting with the structure maps.  An $S^1$-spectrum $E$ is fibrant if and only if it is levelwise fibrant and an $\Omega$-spectrum.  The category $\Spt_k$ is a simplicial model category as well, and we abuse notation slightly by writing $\map(-,-)$ for the simplicial mapping space in this category.  The category $\Spt_k$ is a stable simplicial model category (this follows from the fact that the threefold permutation on $S^1 \wedge S^1 \wedge S^1$ acts as the identity, which one deduces immediately from the corresponding fact for simplicial sets).

If $\mathscr{Y}$ is a pointed space, we write $\Sigma^{\infty}\mathscr{Y}$ for the $S^1$-suspension spectrum attached to $\mathscr{Y}$. There are then standard simplicial Quillen adjunctions
\[
\xymatrix{
\Sigma^{\infty}: \Spc_{k,\bullet} \ar@<+.2em>[r] & \ar@<+.2em>[l] \Spt_{k}: \Omega^{\infty}.
}
\]
If $E$ and $E'$ are $S^1$-spectra, we will abuse notation and write $[E,E']_{\aone} := \hom_{\mathrm{SH}^{S^1}_{\aone}}(E,E')$.  Our next result says that $[U,\mathscr{X}]_{\aone}$ can, under suitable hypotheses, be identified in terms of mappings in the $S^1$-stable homotopy category.

\begin{prop}
\label{prop:stablegroupstructure}
Suppose $k$ is a field, $\mathscr{X}$ is a pointed $k$-space pulled back from a perfect subfield, and fix an integer $n \geq 2$.  If $\mathscr{X}$ is $\aone$-$(n-1)$-connected space, and $U \in \Sm_k$  has $\aone$-cohomological dimension $d \leq 2n-2$, then the map
\[
[U,\mathscr{X}]_{\aone} \longrightarrow [\Sigma^{\infty} U_+,\Sigma^{\infty} \mathscr{X}]_{\aone}
\]
is a bijection, functorial in both inputs.
\end{prop}

\begin{proof}
This result is essentially a compactness argument, and we give a detailed outline of the proof; related results are established in \cite[\S 9]{DuggerIsaksenCellular}.  Let $(E_n)_{n \geq 0}$ be a level-wise fibrant replacement of $\Sigma^{\infty}\mathscr{X}$ (i.e., $E_n$ is a fibrant replacement of $\Sigma^n \mathscr{X}$) and let $E$ be a fibrant replacement of $\Sigma^{\infty} \mathscr{X}$.  Since filtered colimits in $\Spc_{k,\bullet}$ preserve fibrant objects, it follows that there is a simplicial weak equivalence of the form $\Omega^{\infty} E \cong \colim_{n} \Omega^n E_n$.

If $\mathscr{F}$ is a pointed fibrant space, then there is an identification of {\em unpointed} mapping spaces $\map(U,\mathscr{F}) = \F(U)$.  Thus, we conclude that there is an identification of pointed mapping spaces $\map(U_+,\F) = \F(U)$ as well.  Since $U_+$ is $\omega$-compact in $\Spc_{k,\bullet}$ we conclude that
\[
\map(U_+,\colim_n \Omega^n E_n) \cong \colim_n \map(U_+,\Omega^n E_n).
\]
Combining this fact with the evident adjunctions, we see that there is a sequence of homotopy equivalences of Kan complexes of the form
\[
\begin{split}
\map_{\Spt_k}(\Sigma^{\infty} U_+,\Sigma^{\infty} \mathscr{X}) &\cong \map_{\Spt}(\Sigma^{\infty}U_+,E) \\
&\cong \map(U_+,\Omega^{\infty} E) \\
&\cong \map(U_+,\colim_n \Omega^n E_n) \\
&\cong \colim_n \map(U_+,\Omega^n E_n) \\
&\cong \colim_n \map(\Sigma^n U_+,E_n)
\end{split}
\]
Since $\pi_0$ preserves filtered colimits of simplicial sets, we conclude that
\[
\colim_n [\Sigma^n U_+,\Sigma^n \mathscr{X}]_{\aone} \cong [\Sigma^{\infty} U_+,\Sigma^{\infty} \mathscr{X}]_{\aone}.
\]
By adjunction there are functorial identifications $[U_+,\Omega^n \Sigma^n \mathscr{X}]_{\aone} \cong [\Sigma^n U_+,\Sigma^n \mathscr{X}]_{\aone}$ and using Proposition~\ref{prop:groupstructureI} we conclude that the transition maps in $\colim_n [\Sigma^n U_+,\Sigma^n \mathscr{X}]_{\aone}$ are all bijections.
\end{proof}

\subsubsection*{Borsuk's original construction}
In this section, we give the analog of Borsuk's classical construction of the composition on cohomotopy.  If $\mathscr{Y}$ is a pointed space, we use the following notation: $\Delta: \mathscr{Y} \to \mathscr{Y} \times \mathscr{Y}$ is the diagonal map, $\nabla: \mathscr{Y} \vee \mathscr{Y} \to \mathscr{Y}$ is the fold map, and $\mathscr{Y} \vee \mathscr{Y} \to \mathscr{Y} \times \mathscr{Y}$ is the canonical cofibration.

If $\mathscr{X}$ is an $(n-1)$-connected space, $U \in \Sm_k$, and $f,g: U \to \mathscr{X}$ are morphisms, then we can contemplate the following diagram:
\[
\xymatrix{
                  &                                & \mathscr{X} \vee \mathscr{X}  \ar[d]\ar[r]^-{\nabla} & \mathscr{X} \\
U \ar[r]^-{\Delta} & U \times U \ar[r]^-{f \times g} & \mathscr{X} \times \mathscr{X}. &
}
\]
If we can lift the composite $(f \times g) \circ \Delta$ to a morphism $\overline{(f,g)}: U \to \mathscr{\mathscr{X} \vee \mathscr{X}}$, then composing with the fold map, we obtain a map $\nabla \overline{(f,g)}$ that we can think of as the product of $f$ and $g$.  In general, there is an obstruction to producing such a lift, and even if a lift exists it need not be unique.  However, with suitable dimension restrictions imposed on $U$, lifts exist and will be unique.

\begin{prop}
\label{prop:groupstructureII}
Assume $k$ is a field, $\mathscr{X}$ is a pointed $k$-space that is pulled back from a perfect subfield of $k$ and $n \geq 2$ is an integer.  If $\mathscr{X}$ is $\aone$-$(n-1)$-connected space and if $U \in \Sm_k$ has $\aone$-cohomological dimension $d \leq 2n-2$, then composition with the canonical map $\mathscr{X} \vee \mathscr{X} \to \mathscr{X} \times \mathscr{X}$ induces a bijection
\[
[U,\mathscr{X} \vee \mathscr{X}]_{\aone} \longrightarrow [U,\mathscr{X} \times \mathscr{X}]_{\aone}
\]
functorial in both $U$ and $\mathscr{X}$.
\end{prop}

\begin{proof}
It follows from \cite[Corollary 3.3.11]{AsokWickelgrenWilliams} that the map $\mathscr{X} \vee \mathscr{X} \to \mathscr{X} \times \mathscr{X}$ is an $\aone$-$(2n-2)$-equivalence and the argument is a straightforward obstruction theory argument completely parallel to that in Proposition~\ref{prop:groupstructureI}.
\end{proof}


\begin{defn}
\label{defn:secondcomposition}
Assume $k$ is a field, $n \geq 2$ is an integer, $\mathscr{X}$ is a pointed $\aone$-$(n-1)$-connected space pulled back from a perfect subfield of $k$, and $U \in \Sm_k$ has $\aone$-cohomological dimension $d \leq 2n-2$.  Given two elements $f,g \in [U,\mathscr{X}]_{\aone}$, we set $\tau(f,g) := \nabla \overline{(f,g)}$ where $\overline{(f,g)}: U \to \mathscr{X} \vee \mathscr{X}$ is the unique lift of $(f \times g) \circ \Delta: U \to \mathscr{X} \times \mathscr{X}$ guaranteed to exist by \textup{Proposition~\ref{prop:groupstructureII}}.
\end{defn}

\subsubsection*{Comparing the composition operations}
The next result shows that when the composition operations of Propositions~\ref{prop:groupstructureI} and \ref{prop:groupstructureII} are both defined, they coincide.

\begin{prop}
\label{prop:comparingcompositions}
Assume $k$ is a field and fix an integer $n \geq 2$.  If $\mathscr{X}$ is an $\aone$-$(n-1)$-connected space pulled back from a perfect subfield of $k$ and $U \in \Sm_k$ has $\aone$-cohomological dimension $d \leq 2n-2$, then given $f,g \in [U,\mathscr{X}]_{\aone}$, the class $\tau(f,g)$ coincides with the product of $f,g$ from \textup{Proposition~\ref{prop:groupstructureI}}.
\end{prop}

\begin{proof}
If $\mathscr{Y}$ is any pointed $h$-space with multiplication $m$, then the unit condition can be phrased as the existence of a homotopy commutative diagram of the form
\begin{equation}
\label{eqn:A}
\xymatrix{
\mathscr{Y} \vee \mathscr{Y} \ar[r]^-{\nabla}\ar[d]&  \mathscr{Y} \ar@{=}[d] \\
\mathscr{Y} \times \mathscr{Y} \ar[r]_-{m} & \mathscr{Y}.
}
\end{equation}
We now apply this observation with $\mathscr{Y} = \mathscr{X}$ and $\mathscr{Y} = \Omega \Sigma X$ where $\mathscr{X}$ is $\aone$-$(n-1)$-connected and contemplate some obvious diagrams.

First, there is a homotopy commutative diagram of the form
\[
\xymatrix{
\mathscr{X} \vee \mathscr{X} \ar[r]\ar[d] & \Omega \Sigma \mathscr{X} \vee \Omega \Sigma \mathscr{X} \ar[d] \\
\mathscr{X} \times \mathscr{X} \ar[r] & \Omega \Sigma \mathscr{X} \times \Omega \Sigma \mathscr{X}.
}
\]
Thus, if $h \in [U,\mathscr{X} \times \mathscr{X}]_{\aone}$ lifts to $\tilde{h} \in [U,\mathscr{X} \vee \mathscr{X}]_{\aone}$, then the induced class in $[U,\Omega \Sigma \mathscr{X}\times \Omega \Sigma \mathscr{X}]_{\aone}$ lifts to $[U,\Omega \Sigma \mathscr{X} \vee \Omega \Sigma \mathscr{X}]$ as well.

Now, the homotopy commutativity of Diagram~\ref{eqn:A} together with functoriality of the fold map yields a homotopy commutative diagram of the form:
\[
\xymatrix{
\mathscr{X} \vee \mathscr{X} \ar[r]\ar[d] &  \ar[d] \mathscr{X} \\
\Omega \Sigma \mathscr{X} \vee \Omega \Sigma \mathscr{X} \ar[r] \ar[d]& \Omega \Sigma \mathscr{X} \ar[d] \\
\Omega \Sigma \mathscr{X} \times \Omega \Sigma \mathscr{X} \ar[r]^-{m} & \Omega \Sigma \mathscr{X}.
}
\]
Given maps $f,g \in [U,\mathscr{X}]_{\aone}$, since the map $[U,\mathscr{X}]_{\aone} \to [U,\Omega\Sigma \mathscr{X}]_{\aone}$ is a bijection, the homotopy commutativity of the above diagram together with the lifting observation of the previous paragraph show that $\tau(f,g) = m(f,g)$.
\end{proof}

\subsubsection*{Further functoriality properties}
If $\mathscr{X}$ is an $\aone$-$(n-1)$-connected space, setting $\bpi = \bpi_n^{\aone}(\mathscr{X})$, the first non-trivial layer of the $\aone$-Postnikov tower provides an $\aone$-homotopy class of maps $\mathscr{X} \to K(\bpi,n)$.  The space $K(\bpi,n)$ is also an $\aone$-$(n-1)$-connected space.  If $U \in \Sm_k$ has Krull dimension $\leq 2n-2$, then the set $[U,K(\bpi,n)]_{\aone}$ {\em a priori} admits two abelian group structures: one coming from Proposition~\ref{prop:groupstructureI} (i.e., induced by the $h$-group structure on $\Omega \Sigma K(\bpi,n)$) and one from the $h$-space structure of Eilenberg-Mac Lane spaces.  The stabilization map $K(\bpi,n) \to \Omega \Sigma K(\bpi,n)$ is a map of $h$-groups, and thus in the range of dimensions under consideration, the two abelian group structures under consideration necessarily coincide.
The map $\mathscr{X} \to K(\bpi,n)$ yields a homomorphism
\[
[U,\mathscr{X}]_{\aone} \longrightarrow H^n(U,\bpi)
\]
arising from the functoriality clause of Proposition~\ref{prop:groupstructureI}; this homomorphism is often referred to as a Hurewicz homomorphism.  The map $\mathscr{X} \to K(\bpi,n)$ corresponds to a canonically defined cohomology class $\alpha_{\mathscr{X}} \in H^n(\mathscr{X},\bpi)$ (in fact, the coskeletal description of the Postnikov tower \cite[\S 1.2(vi)]{DwyerKan} provides a canonical representing cocycle) that we will refer to as a {\em fundamental class}.

\begin{prop}
\label{prop:functorialityinthetarget}
Assume $k$ is a field, $n \geq 2$ is an integer, $\mathscr{X}$ is a pointed $\aone$-$(n-1)$-connected space pulled back from a perfect subfield of $k$, $\bpi := \bpi_n^{\aone}(\mathscr{X})$, and $U \in \Sm_k$ is a smooth scheme of $\aone$-cohomological dimension $d \leq 2n-2$.
\begin{enumerate}[noitemsep,topsep=1pt]
\item The Hurewicz homomorphism
\[
[U,\mathscr{X}]_{\aone} \longrightarrow H^n(U,\bpi).
\]
is surjective if $d \leq n+1$ and an isomorphism if $d \leq n$.
\item If $f \in [U,\mathscr{X}]_{\aone}$ is an $\aone$-homotopy class of maps, and $\alpha_{\mathscr{X}} \in H^n(\mathscr{X},\bpi)$ is the fundamental class, then the image of $f$ under the Hurewicz homomorphism is $f^*\alpha_{\mathscr{X}}$, where $f^*: H^n(\mathscr{X},\bpi) \to H^n(U,\bpi)$ is the pullback by $f$.
\end{enumerate}
\end{prop}

\begin{proof}
For the first statement, begin by observing that $K(\bpi,n)$ is precisely the $n$-th stage of the $\aone$-Postnikov tower of $\mathscr{X}$.  Then, the obstructions to lifting to the $(n+i)$th stage of the $\aone$--Postnikov tower lie in $H^{n+i+1}(U,\bpi_{n+i}^{\aone}(\mathscr{X}))$, and the space parameterizing lifts are quotients of $H^{n+i}(U,\bpi_{n+i}^{\aone}(\mathscr{X}))$.  If $U$ has $\aone$-cohomological dimension $\leq n+1$, then all the obstruction groups vanish.  If $U$ furthermore has $\aone$-cohomological dimension $\leq n$, then the possible choices of lifts are necessarily unique.  Combining these observations yields the stated surjectivity and bijectivity assertions. 

The second statement follows immediately from the definition of the Hurewicz map and the fundamental class.
\end{proof}

\subsection{Motivic cohomotopy sets made concrete}
\label{ss:cohomotopy}
We now specialize the results of the previous section to the cases of interest in this paper, namely the quadrics $Q_{2n-1}$ and $Q_{2n}$.  In this context, we define motivic cohomotopy sets, observe in Theorem~\ref{thm:naivevsrefined} that these group structures have concrete algebro-geometric interpretations, and then study the Hurewicz homomorphism in great detail (Theorem~\ref{thm:abelian}).

\subsubsection*{Motivic cohomotopy groups}
The following definition is a motivic analog of the classical notion of cohomotopy (we have chosen indexing to agree with the indexing in motivic cohomology); by the observations just made, these cohomotopy sets can be identified as maps into suitable quadrics in certain cases.  We refer the reader to the preliminaries for our conventions regarding spheres.

\begin{defn}[Motivic cohomotopy]
\label{defn:motiviccohomotopy}
Assume $i,j$ are positive integers and $U \in \Sm_k$.  The {\em unstable motivic cohomotopy sets} of $U$ are defined by the formula:
\[
\pi^{i,j}(U) := [U,S^{i,j}]_{\aone}.
\]
The stable motivic cohomotopy groups are defined by the formula:
\[
\pi^{i,j}_s(U) := [\Sigma^{\infty}U_+,\Sigma^{\infty}S^{i,j}]_{\aone}.
\]
\end{defn}

\begin{thm}
\label{thm:excision}
Assume $k$ is a field.   The following statements are true.
\begin{enumerate}[noitemsep,topsep=1pt]
\item If $n \geq 2$, $i-j \geq n$, and $U \in \Sm_k$ has $\aone$-cohomological dimension $d \leq 2n-2$, then the evident map $\pi^{i,j}(U) \to \pi^{i,j}_s(U)$ is a bijection.
\item If $j: W \to X$ is an open immersion and $\varphi: V \to X$ is an \'etale morphism such that the pair $(j,\varphi)$ give rises to a Nisnevich distinguished square, i.e., if $\varphi^{-1}(X \setminus W)_{red} \to (X \setminus W)_{red}$ is an isomorphism, then there is a Mayer-Vietoris long exact sequence of the form
    \[
    \cdots \longrightarrow \pi^{i,j}_s(X) \longrightarrow \pi^{i,j}_s(W) \oplus \pi^{i,j}_s(V) \longrightarrow \pi^{i,j}_s(W \times_X V) \longrightarrow \pi^{i+1,j}_s(X) \longrightarrow \cdots
    \]
\end{enumerate}
\end{thm}

\begin{proof}
Since spheres are defined over $\Spec \Z$ (in particular over the prime field) and $i-j \geq n$, Morel's unstable $\aone$-connectivity theorem \cite[Theorem 6.38]{MField} implies that the space $S^{i,j}$ is at least $\aone$-$(n-1)$-connected.  Bearing this in mind, the first statement is then an immediate consequence of Proposition~\ref{prop:groupstructureI}.

The second statement follows from the existence of a Mayer-Vietoris distinguished triangle:
\[
\Sigma^{\infty}(W \times_X V)_+ \longrightarrow \Sigma^{\infty}(W \sqcup V)_+ \longrightarrow \Sigma^{\infty}X_+.
\]
\end{proof}

\begin{rem}
\label{rem:externalproducts}
There are also external product maps for motivic cohomotopy sets.  For suitable choices of base-point there are pointed $\aone$-weak equivalences of the form $Q_{2m} \wedge Q_{2n} \cong Q_{2(n+m)}$.  Therefore, given morphisms $f: X \to Q_{2m}$ and $g: X \to Q_{2n}$, we can form the composite morphism
\[
X_+ \stackrel{\delta}{\longrightarrow} X_+ \sma X_+ \stackrel{f \sma g}{\longrightarrow} Q_{2m} \wedge Q_{2n} \isomto Q_{2(n+m)}.
\]
This composite defines a functorial morphism $\pi^{2m,m}(X) \times \pi^{2n,n}(X) \longrightarrow \pi^{2(m+n),m+n}(X)$.  The resulting composite depends on the chosen weak-equivalence $Q_{2m} \wedge Q_{2n} \cong Q_{2(n+m)}$ only up to isomorphism, and a straightforward computation of the $\aone$-homotopy class of the map switching factors \cite[Proof of Theorem 4.4.1]{AsokWickelgrenWilliams} shows that the resulting product is $(-\epsilon)$-graded commutative (in $m$ and $n$), where $-\epsilon = -\langle -1 \rangle \in GW(k)$ in the notation of \cite[\S 3.1]{MField}.
\end{rem}

\subsubsection*{The Hurewicz map made concrete}
We now analyze the Hurewicz map of Proposition~\ref{prop:functorialityinthetarget} in the case where $\mathscr{X} = Q_{2n}$.  Indeed, in that case, recall from Proposition~\ref{prop:propertiesofspheres} that $Q_{2n}$ is $\aone$-$(n-1)$-connected, and has $\bpi_n^{\aone}(Q_{2n}) \cong \K^{MW}_n$.  In Lemma~\ref{lem:explicitgenerator}, we gave an explicit generator of the $\K^{MW}_0(k)$-module $H^n(Q_{2n},\K^{MW}_n)$.  Since the fundamental class of Proposition~\ref{prop:functorialityinthetarget}(2) is another module generator of $H^n(Q_{2n},\K^{MW}_n)$, it follows that $\alpha_n$ and $\alpha_{Q_{2n}}$ differ from each other by a constant unit $\lambda \in \K^{MW}_0(k)$.

By the discussion preceding Lemma~\ref{lem:explicitgenerator}, there are canonical and functorial isomorphisms $H^n(U,\K^{MW}_n) \cong \widetilde{CH}^n(U)$ for any smooth $k$-scheme $U$.  Then, combining the isomorphism of Theorem \ref{thm:naivevsrefined} and Proposition~\ref{prop:functorialityinthetarget} the Hurewicz morphism may be viewed as a morphism:
\[
\pi_0(\Singaone Q_{2n})(U) \isomto [U,Q_{2n}]_{\aone} \longrightarrow H^n(U,\K^{MW}_n) \cong \widetilde{CH}^n(U)
\]
sending a morphism $f: U \to Q_{2n}$ to the class $f^*\alpha_n \in \widetilde{CH}^n(U)$.  Putting everything together with Proposition~\ref{prop:groupstructureI} and using the notation of Definition~\ref{defn:motiviccohomotopy}, we may summarize the discussion so far in the following omnibus result.

\begin{thm}
\label{thm:abelian}
Fix an integer $n \geq 2$.  If $k$ is a field having characteristic not equal to $2$ and $U$ is a smooth $k$-scheme of $\aone$-cohomological dimension $d\leq 2n-2$, then the binary operation $\tau$ of \textup{Definition \ref{defn:secondcomposition}} equips $\pi^{2n,n}(X)$ with a functorial structure of abelian group.  The Hurewicz homomorphism
\[
\pi^{2n,n}(U) = [U,Q_{2n}]_{\aone} \longrightarrow \widetilde{CH}^n(U),
\]
has the following properties:
\begin{enumerate}[noitemsep,topsep=1pt]
\item it is surjective if $d \leq n+1$ and an isomorphism if $d \leq n$;
\item it sends a map $f: U \to Q_{2n}$ representing a class in $\pi^{2n,n}(X)$ to $\lambda^{-1} f^*\alpha_n$ where the class $\alpha_n$ differs from the ``fundamental class" $\alpha_{Q_{2n}}$ of $\widetilde{CH}^n(Q_{2n})$ by a constant unit $\lambda \in \K^{\MW}_0(k)^{\times}$.
\end{enumerate}
\end{thm}

\section{Group structures and naive $\aone$-homotopy classes}
\label{s:geometricgroupstructure}
This section is the algebraic and geometric heart of the paper.  Recall from Theorem~\ref{thm:naivevsrefined} that there is an identification $\pi_0(\Singaone Q_{2n})(U) \cong [U,Q_{2n}]_{\aone}$ for $U$ a smooth affine $k$-scheme.  Combining this fact with  Proposition \ref{prop:comparingcompositions} (with $\mathscr{X} = Q_{2n}$), we see that if $U$ has $\aone$-cohomological dimension $\leq 2n-2$, the set $\pi_0(\Singaone Q_{2n})(U)$ inherits a functorial abelian group structure via the operation $\tau$.

In Section~\ref{ss:naivehomotopies}, we review some results of \cite{Fasel15b} and some preliminary moving lemmas that allow us to identify maps $U \to Q_{2n}$ in ``ideal-theoretic" terms.  In Section~\ref{ss:geometricdescription} we use these ideas to give a completely algebraic description of the composition operation $\tau$; this is achieved in Theorem~\ref{thm:concretetau}.

\subsection{Naive homotopies of maps to quadrics}
\label{ss:naivehomotopies}
If $R$ is a ring, then elements of $Q_{2n}(R)$ are sequences $(a_1,\ldots,a_n,b_1,\ldots,b_n,s) \in R^{2n+1}$ satisfying the equation defining $Q_{2n}$; we will write $(a,b,s)$ for such a triple with implicit understanding that $a = (a_1,\ldots,a_n)$, and $b = (b_1,\ldots,b_n)$.  As the notation indicates, we will frequently think of $a \in R^n$ as a row vector, and write $a^t$ for the column vector obtained by transposing $a$.  Given $a,b \in R^n$, and $s \in R$ the assertion that $(a,b,s)$ defines a class in $Q_{2n}(R)$ can equivalently be written as the equality  $ab^t = s(1-s)$.  

If $(a,b,s)$ defines a class in $Q_{2n}(R)$, we can consider the ideal $I := \langle a,s \rangle := \langle a_1,\ldots,a_n,s\rangle \subset R$.  We write $\overline{a}_i$ for the image of $a_i$ under the map $I \to I/I^2$.  The next result shows that, conversely, given an ideal $I$ in $R$ together with lifts of generators of $I/I^2$, there is an induced morphism $\Spec R \to Q_{2n}$.

\begin{lem}[{\cite[Lemma p. 533]{MohanKumarCI}}]
\label{lem:liftexists}
Given an ideal $I$ and a sequence $(a_1,\ldots,a_n)$ of elements in $I$ such that $I/I^2 = \langle \overline{a_1},\ldots,\overline{a_n} \rangle$, there exist an element $s \in I$ and $b_1,\ldots,b_n \in R$ such that $I = \langle a_1,\ldots,a_n,s \rangle$ and $(a,b,s)$ is an element of $Q_{2n}(R)$.
\end{lem}

\begin{proof}
By construction, $M := I/\langle a_1,\ldots,a_n \rangle$ is a finitely generated $R$-module such that $M/IM = 0$.  Nakayama's lemma implies that there exists an element $s \in I$ such that $(1-s)M = 0$.  For any element $m \in I$, we conclude that $m = \sum l_i a_i + ms$ and thus $I = \langle a_1,\ldots,a_n,s \rangle$.  The statement follows by considering the special case $m = s$.
\end{proof}

Using this observation, the naive $\aone$-homotopy class of an element $(a,b,s) \in Q_{2n}(R)$ is essentially determined by $\langle a,s\rangle$; this idea is encoded in the following lemmas.

\begin{lem}
\label{lem:movingb}
Suppose $R$ is a commutative unital ring.  If $a,b,b' \in R^n$ and $s \in R$ are such that $(a,b,s)$ and $(a,b',s)$ are elements of $Q_{2n}(R)$, then $[(a,b,s)] = [(a,b',s)]$ in $\pi_0(\Singaone Q_{2n})(R)$.
\end{lem}

\begin{proof}
We construct an explicit naive homotopy.  Set $B:= b + T(b'-b)$ note that $B(0) = b$ and $B(1) = b'$. Unwinding the definition, we see that $(a,B,s) \in Q_{2n}(R[T])$ gives the required naive homotopy.
\end{proof}

\begin{lem}
\label{lem:movings}
Suppose $R$ is a commutative ring and $I \subset R$ is a finitely generated ideal.  Suppose we are given elements $a_1,\ldots,a_n,s,s' \in I$ such that $(1-s)I \subset \langle a_1,\ldots,a_n \rangle$ and $(1-s')I \subset \langle a_1,\ldots,a_n \rangle$.  If $b,b' \in R^n$ are such that $(a,b,s)$ and $(a,b',s') \in Q_{2n}(R)$, then $[(a,b,s)] = [(a,b',s')]$ in $\pi_0(\Singaone Q_{2n})(R)$.
\end{lem}

\begin{proof}
Set $S := s' + T(s-s') \in R[T]$ so $S(0) = s'$ and $S(1) = s$.  Since $s'$ and $s - s'$ lie in $I$, we see that $S \in I[T]$.  By assumption $\langle a,s \rangle \subset I$ and $(1-s)I \subset \langle a \rangle$; we thus conclude $I = \langle a,s \rangle$.  We claim that $(1-S)I[T] \subset \langle a_1,\ldots,a_n\rangle \subset R[T]$, and to show this it suffices to check that $(1-S)s \in \langle a \rangle \subset R[T]$.

Since $(1-s')I \subset \langle a \rangle$, we know that $(1-s')s \subset \langle a \rangle$ and we may find $c_1,\ldots,c_n \in R$ such that $(1-s')s = ac^t$.  Then, $ab^t - ac^t = s - s^2 - (s - s's) = s(s'-s) \in \langle a \rangle$.  It follows that
\[
(1-S)s = (1 - s' - T(s - s'))s = (1-s')s - Ts(s - s')
\]
lies in $\langle a_1,\ldots,a_n \rangle$ as well.  By appeal to Lemma~\ref{lem:liftexists}, we conclude that there exists $B \in R[T]^n$ such that $(a,B,S) \in Q_{2n}(R[T])$ and $[(a,B(0),s')] = [(a,B(1),s)]$ in $\pi_0(\Singaone Q_{2n})(R)$.

To conclude, we appeal to Lemma~\ref{lem:movingb} to deduce that $[(a,B(1),s)] = [(a,b,s)]$ $\pi_0(\Singaone Q_{2n})(R)$ and also that $[(a,B(0),s')] = [(a,b',s')]$.
\end{proof}

\begin{lem}
\label{lem:naivehomotopyclassdependsonideal}
Suppose $R$ is a commutative ring and $I \subset R$ is a finitely generated ideal.  If $(a,b,s)$ and $(a',b',s')$ determine elements of $Q_{2n}(R)$ satisfying the following properties:
\begin{itemize}[noitemsep,topsep=1pt]
\item[i)] the elements $a_1,\ldots,a_n,a_1',\ldots,a_n' \in I$;
\item[ii)] for $i = 1,\ldots,n$, $a_i - a_i' \in I^2$, and
\item[iii)] $I/I^2 = \langle \overline{a_1},\ldots,\overline{a_n} \rangle = \langle \overline{a_1}',\ldots,\overline{a_n}' \rangle$,
\end{itemize}
then the naive $\aone$-homotopy classes $[(a,b,s)]$ and $[(a',b',s')]$ coincide in $\pi_0(\Singaone Q_{2n})(R)$.
\end{lem}

\begin{proof}
Set $A := (a_1 + (a_1' - a_1)T,\ldots,a_n + (a_n' - a_n)T)$.  Since $I/I^2 = \langle \overline{a_1},\ldots,\overline{a_n} \rangle$, we conclude that the classes of $a_i - (a_i' - a_i)T$ modulo $I[T]^2$ generate $I[T]/(I[T])^2$.  By appeal to Lemma~\ref{lem:liftexists}, we conclude that there exist elements $S \in I[T]$ and $B \in R[T]^n$ such that $(A,B,S) \in Q_{2n}(R[T])$.  Thus, $[(a,B(0),S(0))] = [(a',B(1),S(1))]$ in $\pi_0(\Singaone Q_{2n})(R)$.  The elements $(a,S(0))$ and $(a,s)$ satisfy the hypotheses of Lemma~\ref{lem:movings} and the same statement holds for $(a',S(1))$ and $(a',s')$.  Thus, we conclude that these classes are also naively homotopic and conclude.
\end{proof}

\subsubsection*{Moving maps with target $Q_{2n}$}
In order to describe the abelian group structure on $\pi_0(\Singaone Q_{2n}(R))$ geometrically, it will be helpful to be able to move maps to $Q_{2n}$ into ``general position".  We now describe a procedure to do this, inspired by \cite[Lemma 4.3]{Mandal10} (or \cite[Corollary 2.14]{Bhatwadekar00} when $n=\mathrm{dim}(R)$).  If $(a,b,s) \in Q_{2n}(R)$ as above, then the ideal $\langle a,s \rangle$ need not have height $n$, however, Lemma~\ref{lem:moving} will demonstrate that $(a,b,s)$ is naively $\aone$-homotopic to an element of $Q_{2n}(R)$ for which the associated ideal {\em does} have height $n$.

\begin{lem}
\label{lem:moving}
Suppose $R$ is a Noetherian ring and $(a,b,s) \in Q_{2n}(R)$.  If $J_1,\ldots,J_r\subset R$ are ideals such that $\mathrm{dim}(R/J_i)\leq n-1$ for $i=1,\ldots,r$, then, there exists a sequence $\mu=(\mu_1,\ldots,\mu_n)\in R^n$ and an ideal
\[
N:=\langle a+\mu (1-s)^2,s+\mu b^t(1-s)\rangle,
\]
such that the following statements hold:
\begin{enumerate}[noitemsep,topsep=1pt]
\item the sequence $(a+(1-s)^2\mu,b(1-\mu b^t),s+\mu b^t(1-s))$ yields an element of $Q_{2n}(R)$;
\item the naive $\aone$-homotopy classes $[(a,b,s)]$ and $[(a+\mu (1-s)^2,b(1-\mu b^t),s+\mu b^t(1-s))]$ coincide;
\item $\mathrm{ht}(N)\geq n$; and
\item $J_i+N=R$ for $i=1,\ldots,r$.
\end{enumerate}
\end{lem}

\begin{proof}
We appeal to the results of Eisenbud-Evans \cite{Eisenbud73}, as generalized by Plumstead \cite[p. 1420]{Plumstead83}. To this end, let $\mathcal A\subset \Spec R$ be the set of prime ideals $\mathfrak p\subset R$ such that $(1-s)\not\in \mathfrak p$ and $\mathrm{ht}(\mathfrak p)\leq n-1$. Let moreover $\mathcal B_i:=V((J_i)_{(1-s)})\subset \Spec R_{(1-s)}\subset \Spec R$ for $i=1,\ldots,r$ and $\mathcal B=\cup_i \mathcal B_i$.  Since $R$ is Noetherian, the restriction of the usual dimension function on $\Spec (R_{(1-s)})$ to $\mathcal{A}$ is a generalized dimension function $d: \mathcal{A} \to \N$ in the sense of \cite[Definition p. 1419]{Plumstead83} (cf. \cite[Example 1]{Plumstead83}).  Likewise, let  $d_i:\mathcal B_i\to \N$ be the usual dimension function on $V((J_i)_{(1-s)})$.  As in \cite[Example 2]{Plumstead83}, we obtain a generalized dimension function $\delta:\mathcal A\cup \mathcal B\to \N$ such that $\delta(\mathfrak p)\leq n-1$ for any $\mathfrak p\in \mathcal A\cup \mathcal B$.

Now, we apply the Eisenbud-Evans theorems to the finitely generated free $R$-module $R^{n}$ with the generalized dimension function $\delta$ on $\mathcal{A} \cup \mathcal{B}$.  Then, if $(a,(1-s)^2)$ is unimodular at $\mathcal A\cup \mathcal B$, we conclude that there exists a sequence $\mu=(\mu_1,\ldots,\mu_n)\in R^n$ such that the row  $(a_1+\mu_1(1-s)^2,\ldots,a_n+\mu_n(1-s)^2)$ is unimodular at $\mathcal A\cup \mathcal B$ (a priori the Eisenbud-Evans results are formulated in terms of basic elements, however the two notions coincide for finitely generated free modules \cite[Lemma 1]{Eisenbud73}).

To establish points (1) and (2), observe that if we set $A=a+T(1-s)^2\mu\in R[T]^n$, then
\[
Ab^t=ab^t+T(1-s)^2\mu b^t=s(1-s)+T(1-s)^2\mu b^t=(1-s)-(1-s)^2(1-T\mu b^t).
\]
Multiplying both sides by $(1-T\mu b^t)$, we obtain the equality
\[
Ab^t(1-T\mu b^t)=(1-s)(1-T\mu b^t)-(1-s)^2(1-T\mu b^t)^2.
\]
Setting $B=(1-T\mu b^t)b$, one deduces that $(A,B,(1-s)(1-T\mu b^t))\in Q_{2n}(R[T])$. It follows that $(A,B,1-(1-s)(1-T\mu b^t))=(A,B,s+T\mu b^t(1-s))\in Q_{2n}(R[T])$.  The first two points then follow by evaluating the homotopy at $T = 0,1$.

To establish points (3) and (4), observe that
\begin{eqnarray*}
\langle a+\mu (1-s)^2\rangle & = & \langle a+\mu (1-s)^2,(1-s)(1-\mu b^t)\rangle\cap \langle a+\mu (1-s)^2,s+\mu b^t(1-s)\rangle \\
 & = & \langle a+\mu (1-s)^2,(1-s)(1-\mu b^t)\rangle\cap N.
\end{eqnarray*}
If $\mathfrak p$ is a prime ideal such that $N\subset \mathfrak p$, then it follows that $s+\mu b^t(1-s)\in \mathfrak p$ and therefore that $1-(s+\mu b^t(1-s))=(1-s)(1-\mu b^t)\not\in \mathfrak p$. Consequently, $(1-s)\not\in \mathfrak p$. Moreover, $\langle a+\mu (1-s)^2\rangle\subset N\subset \mathfrak p$. As $(a_1+\mu_1(1-s)^2,\ldots,a_n+\mu_n(1-s)^2)$ is unimodular at $\mathcal A\cup \mathcal B$, it follows that $\mathfrak p\not \in \mathcal A\cup \mathcal B$. Therefore, $\mathfrak p\not\in \cup_i V(J_i)$ and condition (4) follows. Similarly, $\mathfrak p\not\in \mathcal A$ and condition (3) is also satisfied.
\end{proof}

\subsection{A geometric description of the composition on cohomotopy groups}
\label{ss:geometricdescription}
In order to describe the group structure on $[U,Q_{2n}]_{\aone}$ explicitly, it is convenient to use the description of the product given in Definition~\ref{defn:secondcomposition}; recall that this composition coincides with that described in Proposition~\ref{prop:groupstructureI} by appealing to Proposition~\ref{prop:comparingcompositions}.  More precisely, given $f,g: U \to Q_{2n}$, we consider the following diagram
\[
\xymatrix{
                  &                                 &  Q_{2n} \vee Q_{2n} \ar[r]^-{\nabla}\ar[d] & Q_{2n} \\
U \ar[r]^-{\Delta} &  U \times U \ar[r]^-{f \times g} &  Q_{2n} \times Q_{2n} &
}
\]
In order to make these constructions more explicit, we will provide a model of $Q_{2n} \vee Q_{2n}$ as a smooth affine scheme.  Given such a model, appeal to Theorem~\ref{thm:naivevsrefined} allows us to deduce that the maps $Q_{2n} \vee Q_{2n} \to Q_{2n} \times Q_{2n}$ and $Q_{2n} \vee Q_{2n} \to Q_{2n}$ may be represented by explicit morphisms of smooth affine schemes, unique up to naive $\aone$-homotopy.

\subsubsection*{A smooth model of $Q_{2n} \vee Q_{2n}$}
As discussed before Lemma~\ref{lem:explicitgenerator}, recall that $Z_n$ is the closed subscheme of $Q_{2n}$ defined by $x_1 = \cdots = x_n = z = 0$ and the open complement $X_{2n} := Q_{2n} \setminus Z_{2n}$ is $\aone$-contractible by \cite[Theorem 3.1.1]{ADF}.  Since $X_{2n}$ is $\aone$-contractible, the inclusion of the base-point in $X_{2n}$ yields a commutative diagram of smooth $k$-schemes of the form
\[
\xymatrix{
	Q_{2n} \ar[d]& \ar[l] \ast \ar[r]\ar[d]& Q_{2n} \ar[d] \\
	Q_{2n} \times X_{2n} & \ar[l] X_{2n} \times X_{2n} \ar[r]& X_{2n} \times Q_{2n},
}
\]
where all vertical morphisms are $\aone$-weak equivalences.  In particular, the evident map of homotopy colimits is also an $\aone$-weak equivalence.  

Since all the horizontal morphisms in this diagram are cofibrations, the homotopy colimit of each row coincides with the actual colimit.  The colimit of the top row is, by definition, the wedge sum.  On the other hand, the colimit of the bottom row is simply the union in $Q_{2n} \times Q_{2n}$ of $Q_{2n} \times X_{2n}$ and $X_{2n} \times Q_{2n}$, i.e., it is the open subscheme of the product $Q_{2n} \times Q_{2n}$ whose closed complement is $Z_n \times Z_n$.  The induced map of colimits thus yields a morphism $Q_{2n} \vee Q_{2n} \to (Q_{2n} \times Q_{2n})\setminus (Z_n \times Z_n)$ that we will refer to as the natural inclusion.  The next result is an immediate consequence of these observations. 

\begin{lem}
\label{lem:modelforwedgesum}
Fix a base-point in $X_{2n}$ and point $Q_{2n}$ by its image.   The coproduct-to-product map $Q_{2n} \vee Q_{2n} \to Q_{2n} \times Q_{2n}$ factors as
\[
Q_{2n} \vee Q_{2n} \longrightarrow (Q_{2n} \times Q_{2n})\setminus (Z_n \times Z_n) \longrightarrow Q_{2n} \times Q_{2n};
\]
where the first morphism is the natural inclusion and is an $\aone$-weak equivalence, and the second morphism is an open immersion.
\end{lem}

\subsubsection*{A smooth affine model of $Q_{2n} \vee Q_{2n}$}
Granted that $Q_{2n} \vee Q_{2n}$ admits an explicit model as a smooth scheme, we may use the Jouanolou trick to obtain a smooth affine model. More precisely, the Jouanolou trick shows that there exists an affine vector bundle torsor $\phi: \widetilde{Q_{2n} \vee Q_{2n}} \to (Q_{2n} \times Q_{2n}) \setminus (Z_n \times Z_n)$  \cite[Definition 4.2]{WeibelHAK}.  The following example gives a general construction of Jouanolou devices.

\begin{construction}
\label{construction:explicitjouanolou}
Suppose $X$ is any regular affine scheme, and $Z \subset X$ is a closed subscheme equipped with a choice $f_1,\ldots,f_r$ of generators of the ideal corresponding to $Z$.  There is an induced morphism $f: X \to {\mathbb A}^r$ such that $f^{-1}(0) = Z$.  Pulling back the morphism $Q_{2r-1} \to {\mathbb A}^r \setminus 0$ (which one can check is a torsor under a trivial vector bundle of rank $r-1$) along $f$ one obtains a torsor under a vector bundle $\widetilde{X \setminus Z} \to X \setminus Z$.  More explicitly, the scheme $\widetilde{X \setminus Z}$ is the closed subscheme of $X \times {\mathbb A}^r$ defined by the equation $\sum_i f_i y_i = 1$ and the map to $X \setminus Z$ is induced by projection onto the first factor.

We now take $X = Q_{2n} \times Q_{2n}$ and $Z = Z_n \times Z_n$, which has codimension $2n$ in $X$, but is cut out by $2n+2$-equations.  We embed $Q_{2n} \times Q_{2n}$ in ${\mathbb A}^{2n+1} \times {\mathbb A}^{2n+1}$ with coordinates $(x_1,\ldots,x_n,y_1,\ldots,y_n,z) = (x,y,z)$ and $(x_1',\ldots,x_n',y_1',\ldots,y_n',z') = (x',y',z')$.  If we give ${\mathbb A}^{2n+2}$ the coordinates
\[
(u_1,\ldots,u_n,u_{n+1},u_1',\ldots,u_n',u_{n+1}') = (u,u_{n+1},u',u_{n+1}'),
\]
the construction above yields a closed subscheme of $X \times X\times {\mathbb A}^{2n+2}$ with coordinates $(x,y,z,x',y',z',u,u_{n+1},u',u'_{n+1})$ and coordinate ring:
\[
\begin{split}k[x,y,z,x',y',z',u,u_{n+1},u',u'_{n+1}]/\langle &xy^t = z(1-z), \\ &x'(y')^t = z'(1-z'),\\  &ux^t + u'(x')^t + u_{n+1}z + u'_{n+1}z' = 1\rangle \end{split}.
\]
The projection morphism for the Jouanolou device in these coordinates is induced by the projection onto $x,y,z,x',y',z'$; in particular, it has relative dimension $2n+1$.  Henceforth, $\widetilde{Q_{2n} \vee Q_{2n}}$ will be the Jouanolou device just described.
\end{construction}

Recall that $Q_{2n}$ is pointed with base point $(0,\ldots,0,0\ldots,0,1)$ and thus the product has an induced base-point.  The maps $Q_{2n} \to Q_{2n} \times Q_{2n}$ obtained by inclusion of the base-point in one component factor through closed immersions $i_l,i_r: Q_{2n} \to (Q_{2n} \times Q_{2n})\setminus (Z_n \times Z_n)$.  We may lift $i_l$ and $i_r$ through morphisms $\tilde{i_l},\tilde{i_r}: Q_{2n} \to \widetilde{Q_{2n} \vee Q_{2n}}$ that agree on the base-point.  Explicitly, we may define $\tilde{i_l}$ in the coordinates $x,y,z,x',y',z',u,u_{n+1},u',u'_{n+1}$ by setting $x' = y' = u = u' = 0$, $u_{n+1} = 0$, and $z' = u'_{n+1} = 1$.  Similarly, we define $\tilde{i_r}$ by setting $x = y = u = u' = 0$, $u'_{n+1}=0$ and $z = u_{n+1} = 1$.

\begin{rem}
The pullback of a torsor under a vector bundle is a torsor under a vector bundle, and since torsors under vector bundles on affine schemes are trivial, they become vector bundles after choice of a section.  Thus, the pullback of $\phi$ along either $i_l$ or $i_r$ is isomorphic to a vector bundle of rank $2n+1$ on $Q_{2n}$ and by suitably choosing sections we may obtain the required lifts.
\end{rem}

\begin{lem}
\label{lem:naivehomotopieslifttojouanolou}
If
\[
p: \widetilde{Q_{2n} \vee Q_{2n}} \longrightarrow (Q_{2n} \times Q_{2n})\setminus (Z_n \times Z_n)
\]
is the projection morphism in Construction~\ref{construction:explicitjouanolou}, then for any smooth affine $k$-scheme $U$ the induced map $\pi_0(\Singaone \widetilde{Q_{2n} \vee Q_{2n}}(U)) \to \pi_0(\Singaone (Q_{2n} \times Q_{2n}) \setminus (Z_n \times Z_n)(U))$ is a bijection.  In particular, any morphism of $k$-schemes $h: U \to (Q_{2n} \times Q_{2n}) \setminus (Z_n \times Z_n)$ lifts uniquely up to naive $\aone$-homotopy to a morphism of $k$-schemes $\tilde{h}: U \to \widetilde{Q_{2n} \vee Q_{2n}}$.
\end{lem}

\begin{proof}
The map $p$ is an affine vector bundle torsor by construction, and therefore the result is a special case of \cite[Lemma 4.2.4]{AsokHoyoisWendtII}.
\end{proof}

\subsubsection*{A scheme-theoretic model for the fold map $Q_{2n} \vee Q_{2n} \to Q_{2n}$}
We now describe a model of the fold map that admits a nice interpretation in ideal-theoretic terms; the explicit formulas we write down will be useful later.  Since $\widetilde{Q_{2n} \vee Q_{2n}}$ is a smooth affine $k$-scheme, the class of the fold map in $\nabla \in [Q_{2n} \vee Q_{2n},Q_{2n}]_{\aone}$ is, by means of Theorem~\ref{thm:naivevsrefined}, represented by a unique up to $\aone$-homotopy morphism of smooth $k$-schemes $\nabla: \widetilde{Q_{2n} \vee Q_{2n}} \to Q_{2n}$.

\begin{construction}
\label{construction:geometricfold}
We use the smooth affine scheme $\widetilde{Q_{2n} \vee Q_{2n}}$ from Construction~\ref{construction:explicitjouanolou}.  We now describe a morphism whose $\aone$-homotopy class coincides with that of the fold map.  Our construction is based on the idea made precise in Lemma~\ref{lem:naivehomotopyclassdependsonideal} that the naive $\aone$-homotopy class of a morphism $\varphi: \Spec R \to Q_{2n}$ defined by $(a,b,s)$ is essentially determined by the ideal $\langle a,s \rangle$.

Consider the ideals $I := \langle x, z \rangle$ and $I' = \langle x',z' \rangle$.  Set $J := II'$, i.e., $J$ is the ideal generated by $x_ix_j'$, $x_iz'$, $x_i'z$ and $zz'$ as $i$ and $j$ range from $1$ through $n$.

The equation $ux^t+u_{n+1}z+u'(x^\prime)^t+ u'_{n+1} z^\prime=1$ implies that the ideals $I$ and $I'$ are comaximal and therefore $II' = I \cap I^\prime$.  The Chinese remainder theorem guarantees the existence of an isomorphism
\[
J/J^2\simeq J/IJ\times J/I^\prime J.
\]
We claim that the inclusions $J\subset I$ and $J \subset I'$ induce isomorphisms $J/IJ \isomt I/I^2$ and $J/I'J \isomt I'/I'^2$.  We treat the case $J/IJ \to I/I^2$; define a homomorphism $I \to J/IJ$ by sending $a \in I$ to the image of $a (u'(x')^t + u_{n+1}'z')$ in $J/IJ$.  One checks that this homomorphism factors through a homomorphism $I/I^2 \to J/IJ$ inverse to the homomorphism $J/IJ \to I/I^2$ induced by $J \subset I$.

Now, set
\[
c_i:= x_i^\prime (ux^t+u_{n+1}z)^2+x_i(u'(x^\prime)^t+ u'_{n+1} z^\prime)^2.
\]
We claim that $J/J^2=(\overline c_1,\ldots,\overline c_n)$.  By means of the identification $J/J^2 \cong J/IJ \times J/I'J \cong I/I^2 \times I'/(I')^2$, it suffices to show that the images of $\overline{c}_i$ under the isomorphism of the previous paragraph generate $I/I^2 \times I'/(I')^2$.  To this end, observe that
\[
\bar{c}_i = \overline{x_i(u'(x')^t + u_{n+1}'z')^2} \bmod IJ.
\]
Since $x_i(u'(x')^t + u_{n+1}'z') \in I^2$, we see that the image of $\bar{c}_i$ in $I/I^2$ coincides with $\bar{x}_i$.  Similarly, the image of $\bar{c}_i$ in $I'/(I')^2$ is $\bar{x}'_i$.  Likewise, one sees that $(ux^t + u_{n+1}z)c_i$ is sent to $\bar{x}'_i$ in $I'/(I')^2$, while its class in $IJ$ is trivial and a similar statement holds for $(u'(x')^t + u_{n+1}'z')c_i$.

We saw above that $c \equiv x \bmod I^2$ and $c \equiv x' \bmod (I')^2$.  Therefore, $I \equiv \langle c \rangle + I^2$ and $I' \equiv \langle c \rangle + {I'}^2$.  By appeal to Lemma~\ref{lem:liftexists} we conclude that:
\begin{itemize}[noitemsep,topsep=1pt]
\item[i)] there exist elements $w \in I$ and $w' \in I'$ such that $I = \langle c,w \rangle$ and $I' = \langle c,w' \rangle$, and
\item[ii)] there exist $n$-tuples of regular functions $d,d'$ on the explicit Jouanolou device such that $w(1-w) = cd^t$ and $w'(1-w') = c{d'}^t$.
\end{itemize}
Then, one can check $J = \langle c,ww' \rangle$ and the equation $ww'(1-ww') = c \delta^t$ is satisfied with $\delta = (c(d^\prime)^t)d+(w^\prime)^2d+w^2d^\prime$. The sequence $(c,\delta,ww')$ corresponds to a morphism $\nabla$ from the Jouanolou device to $Q_{2n}$.
\end{construction}

The next result follows immediately from Construction~\ref{construction:geometricfold} by observing that the restriction of the morphism $\widetilde{Q_{2n} \vee Q_{2n}} \to Q_{2n}$ defined by $(c,\delta,ww')$ along either closed immersion $\tilde{i_l}$ or $\tilde{i_r}$ is the identity map $Q_{2n} \to Q_{2n}$.

\begin{lem}
\label{lem:geometricfold}
The map $\nabla: \widetilde{Q_{2n} \vee Q_{2n}} \to Q_{2n}$ described in \textup{Construction~\ref{construction:geometricfold}} is a model for the fold map.
\end{lem}

\subsubsection*{A geometric lift}
Suppose that $U$ is a smooth affine $k$-scheme of dimension $d \leq 2n-2$ and $f,g: U \to Q_{2n}$.  We consider the map $(f \times g)\circ \Delta: U \to Q_{2n} \times Q_{2n}$.  While $(f \times g)\circ \Delta$ does not necessarily factor through $(Q_{2n} \times Q_{2n}) \setminus (Z_n \times Z_n)$, we now show that, up to replacing $f$ and $g$ by naively $\aone$-homotopic maps, such a lift does exist.  In fact, we will establish slightly more: we will show that we can choose $f'$ and $g'$ such that $(f')^{-1}(Z_n)$ and $(g')^{-1}(Z_n)$ are disjoint, i.e., the corresponding ideals in $U$ are comaximal.

\begin{lem}
\label{lem:generalpositionmaps}
Suppose $U$ is a smooth affine $k$-scheme of dimension $d \leq 2n-2$ and $f,g: U \to Q_{2n}$.
\begin{enumerate}[noitemsep,topsep=1pt]
\item There exist $f',g': U \to Q_{2n}$ such that $f$ is naively $\aone$-homotopic to $f'$, $g$ is naively $\aone$-homotopic to $g'$ and the morphism $(f' \times g')\circ \Delta$ factors through $(Q_{2n} \times Q_{2n}) \setminus (Z_n \times Z_n)$.
\item If $f''$ and $g''$ are another pair as in Point (1), then the $\aone$-homotopy classes of the morphisms $[(f'' \times g'')\circ \Delta]$ and $[(f' \times g')\circ \Delta]$ coincide in $[U,Q_{2n} \vee Q_{2n}]_{\aone}$ (under the $\aone$-weak equivalence of \textup{Lemma~\ref{lem:modelforwedgesum}}).
\end{enumerate}
\end{lem}

\begin{proof}
Suppose $U = \Spec R$ for some $k$-algebra $R$.  Morphisms $f,g: U \to Q_{2n}$ correspond to sequences $(a,b,s)$, $a = (a_1,\ldots,a_n)$, $b = (b_1,\ldots,b_n)$ and $(c,d,t)$, $c=(c_1,\ldots,c_n)$, $d = (d_1,\ldots,d_n)$ of elements of $R$.

Appealing to Lemma~\ref{lem:moving} with all $J_i$ taken equal to the unit ideal we may replace $g$ by a naively $\aone$-homotopic map $g'$, such that the associated ideal $I_2'$ has height $\geq n$.  Next, we want to show that by replacing $f$ by a naively $\aone$-homotopic map if necessary, we may assume that $f^{-1}(Z_n)$ and $g^{-1}(Z_n)$ are disjoint subschemes, i.e., that $f\times g$ misses $Z_n\times Z_n$ in $Q_{2n}\times Q_{2n}$.  To this end, set $J_1 = I_2$ (and take $J_i$ to be the unit ideal otherwise).  Since the height of $J_1$ is assumed $\geq n$, we see that $\dim R/J_1 \leq 2n-2 - n = n-2$.  The existence of the required homotopy is then guaranteed by Lemma~\ref{lem:moving}.

Since $f$ is naively $\aone$-homotopic to $f'$ and $g$ is naively $\aone$-homotopic to $g'$, we conclude that $(f \times g)\circ \Delta$ is naively $\aone$-homotopic to $(f' \times g')\circ \Delta$.  Therefore, if $f''$ and $g''$ are another such pair, the fact that $[(f' \times g')\circ \Delta]$ and $[(f'' \times g'')\circ \Delta]$ coincide in $[U,Q_{2n} \vee Q_{2n}]_{\aone}$ follows immediately by appeal to Proposition~\ref{prop:groupstructureII}.
\end{proof}

Consider the following diagram:
\[
\xymatrix{
& & \widetilde{Q_{2n} \vee Q_{2n}} \ar[r]^-{\nabla}\ar[d]^{p} & Q_{2n}; \\
& & (Q_{2n} \times Q_{2n}) \setminus (Z_n \times Z_n)\ar[d] & \\
U \ar[r]^-{\Delta}& U \times U \ar[r]^-{f \times g} & Q_{2n} \times Q_{2n}, &
}
\]
where the morphisms $p$ and $\nabla$ in this diagram are described in Lemmas~\ref{lem:naivehomotopieslifttojouanolou} and \ref{lem:geometricfold}.  Lemma~\ref{lem:generalpositionmaps} shows that, possibly after replacing $f$ and $g$ by naively $\aone$-homotopic maps $f'$ and $g'$, the composite $(f \times g)\circ \Delta$ factors through $(Q_{2n} \times Q_{2n}) \setminus (Z_n \times Z_n)$.  Lemma~\ref{lem:naivehomotopieslifttojouanolou} then shows that such a morphism lifts uniquely up to naive $\aone$-homotopy through a morphism $\widetilde{(f' \times g')\circ \Delta}: U \to \widetilde{Q_{2n} \vee Q_{2n}}$.  The composite $\nabla \circ \widetilde{(f' \times g')\circ \Delta}$ thus defines a morphism $U \to Q_{2n}$.  By tracing through construction Construction~\ref{construction:geometricfold}, we will now give an algebraic description of the naive homotopy class of $\nabla \circ \widetilde{(f' \times g')\circ \Delta}$.

\begin{construction}
\label{construction:geometriccomposition}
Suppose $U=\Spec R$ is a smooth affine $k$-scheme of dimension $d \leq 2n-2$ and $f,g: U\to Q_{2n}$.  Assume that $f$ and $g$ correspond with sequences $v=(a,b,s)$ and $v^\prime=(a^\prime,b^\prime,s^\prime)\in Q_{2n}(R)$ such that $I(v)=\langle a,s\rangle$ and $I(v^\prime)=\langle a^\prime,s^\prime\rangle$ are comaximal. Since $I(v)$ and $I(v^{\prime})$ are comaximal, we conclude that $I(v)I(v^\prime)=I(v)\cap I(v^\prime)$ and we set:
\[
J:=I(v)I(v^\prime)=I(v)\cap I(v^\prime).
\]

As $J/J^2\cong I(v)/I(v)^2\times I(v^\prime)/I(v^\prime)^2$, we see that there exist elements $c_1,\ldots,c_n\in J$ such that $c=(c_1,\ldots,c_n)$ satisfies $c\equiv a\pmod {I(v)}$, $c\equiv a^\prime\pmod {I(v^\prime)}$ and $J=\langle c\rangle \pmod{J^2}$. In particular, the following hold:
\begin{eqnarray*}
I(v)& = & \langle c\rangle + I(v)^2, \text{ and } \\
I(v^\prime) & = & \langle c\rangle + I(v^\prime)^2.
\end{eqnarray*}
By appeal to Lemma~\ref{lem:liftexists},  we conclude the following:
\begin{itemize}[noitemsep,topsep=1pt]
\item[i)] there exist elements $u\in I(v)$ and $u^\prime\in I(v^\prime)$ such that $I(v)=\langle c,u\rangle$ and $I(v^\prime)=\langle c,u^\prime\rangle$,
\item[ii)] there exist elements $d,d^\prime \in R^n$ such that the equations $u(1-u)=cd^t$ and $u^\prime(1-u^\prime)=c(d^\prime)^t$
  are satisfied.
\end{itemize}
Using these relations, we can write $J=\langle c,uu^\prime\rangle$ and the equation $uu^\prime(1-uu^\prime)=cx^t$, with $x=(c(d^\prime)^t)d+(u^\prime)^2d+u^2d^\prime$ is satisfied.   The $(2n+1)$-uple $(c,x,uu^\prime)$ yields a morphism $h:U\to Q_{2n}$.
\end{construction}

\begin{lem}
\label{lem:welldefined}
Suppose $k$ is a field having characteristic different from $2$ and $n \geq 2$ is an integer.  Suppose $U = \Spec R$ is a smooth affine $k$-scheme of dimension $d \leq 2n-2$ and we are given two morphisms $f,g: U \to Q_{2n}$.  There is a function
\[
\begin{split}
\tau: \pi_0(\Singaone Q_{2n}(R)) \times \pi_0(\Singaone Q_{2n}(R)) &\longrightarrow \pi_0(\Singaone Q_{2n}(R)), \\
([f],[g])&\longmapsto [h]
\end{split}
\]
defined as follows: choose $f'$ and $g'$ naively $\aone$-homotopic to $f$ and $g$ and with $(f')^{-1}(Z_n)$ and $(g')^{-1}(Z_n)$ disjoint and send $(f',g')$ to the output $h$ of \textup{Construction~\ref{construction:geometriccomposition}}.
\end{lem}

\begin{proof}
By appeal to Lemma~\ref{lem:generalpositionmaps}, we may always suppose that $(f')^{-1}(Z_n)$ and $(g')^{-1}(Z_n)$ are disjoint.  We now trace through Construction~\ref{construction:geometriccomposition} to see how the output depends on the chosen representatives.

Following the notation of  Construction~\ref{construction:geometriccomposition}, the element $h$ depends on the choices of elements $c,x \in R^n$ and $u,u' \in R^n$.  Pick $c' \in R^n$ such that $c' \equiv a (\bmod I(v))$, $c' \equiv a' (\bmod I(v'))$ and $J = \langle c' \rangle (\bmod J^2)$.  In that case, $I(v) = \langle c' \rangle + I(v)^2$ and $I(v') = \langle c ' \rangle + I(v')^2$.  From this we conclude that $c_i - c_i' \in I(v)^2$ and also that $c_i - c_i' \in I(v')^2$.

As in Construction~\ref{construction:geometriccomposition}, we build elements $\mu \in I(v)$ and $\mu' \in I(v')$ such that $I(v) = \langle c',\mu \rangle$ and $I(v') = \langle c', \mu' \rangle$ and elements $\delta, \delta' \in R^n$ such that $\mu(1-\mu) = c'\delta^t$ and $\mu'(1-\mu') = c'(\delta')^t$.  Then, $J = \langle c',\mu\mu' \rangle$ and by setting $x' = c' (\delta')^t \delta + (\mu')^2\delta+\mu^2\delta'$, the equation $\mu\mu'(1-\mu\mu') = c'(x')^t$ is satisfied.  By appeal to Lemma~\ref{lem:naivehomotopyclassdependsonideal}, one concludes that $(c,x,uu')$ and $(c',x',\mu\mu')$ yield the same class in $\pi_0(\Singaone Q_{2n}(R))$, irrespective of the choice of $\mu$, $\mu'$ and $x'$.
\end{proof}

Finally, we compare the output of Lemma~\ref{lem:welldefined} with the composition operation in $\pi_0(\Singaone Q_{2n}(R))$ defined in Proposition~\ref{prop:groupstructureI}.

\begin{thm}
\label{thm:concretetau}
Assume $k$ is a field having characteristic not eequal to $2$ and $U = \Spec R$ is a smooth affine $k$-scheme of dimension $d \leq 2n-2$.  The composition $[f],[g] \mapsto \tau([f],[g])$ on $\pi_0(\Singaone Q_{2n}(R))$ given in \textup{Lemma~\ref{lem:welldefined}} coincides with the operation from \textup{Proposition~\ref{prop:groupstructureI}}.
\end{thm}

\begin{proof}
By Proposition~\ref{prop:comparingcompositions}, it suffices to prove that the composition $\tau$ is the same as that described in Definition~\ref{defn:secondcomposition}.

Suppose $f$ and $g$ are such that $f^{-1}(Z_n)$ and $g^{-1}(Z_n)$ are disjoint, which we can assume by Lemma~\ref{lem:generalpositionmaps}. In that case, we know that $(f \times g) \circ \Delta$ lifts, up to naive $\aone$-homotopy, to a morphism $h: U \to (Q_{2n} \times Q_{2n}) \setminus (Z_n \times Z_n)$.  By appeal to Lemma~\ref{lem:naivehomotopieslifttojouanolou}, the morphism $h$ lifts, uniquely up to naive $\aone$-homotopy to a morphism $\tilde{h}: U \to \widetilde{Q_{2n} \vee Q_{2n}}$.

We can write down an ``explicit" formula for the map $\tilde{h}$, at least up to naive $\aone$-homotopy.  Let $\varphi:k[x,y,z]/(xy^t=z(1-z))\to k[U]$ and $\psi:k[x,y,z]/(xy^t=z(1-z))\to k[U]$ be the maps corresponding to $f$ and $g$.  We write down a ring map from the coordinate ring displayed in Construction~\ref{construction:explicitjouanolou}.  The images of the variables $(x,y,z,x^\prime,y^\prime,z^\prime)$ are determined by $\varphi$ and $\psi$. By assumption, we know that the ideals $I:=\langle \varphi(x),\varphi(z)\rangle$ and $I^\prime:= \langle \psi(x^\prime),\psi(z^\prime)\rangle$ are comaximal. Therefore we can find $n+1$-tuples $(a,a_{n+1})$ (where, as above $a = (a_1,\ldots,a_n)$) and $(b,b_{n+1})$ such that
\[
a\varphi(x)^t + a_{n+1}\varphi(z) + b \psi(x')^t + b_{n+1} \psi(z') = 1.
\]
Then, sending $u\mapsto a$ and $v\mapsto b$, we obtain the required morphism.

Thus, to establish the result, it suffices to show that $[h]$ as in Lemma~\ref{lem:welldefined} is a model for the composite of this lift and the geometric $\nabla$.  Given the constructions above, the result follows by comparison with the explicit formula for $\nabla$ given in Construction~\ref{construction:geometricfold} combined with Lemma~\ref{lem:geometricfold}.
\end{proof}

\begin{rem}
The formula in Theorem~\ref{thm:concretetau} is motivated by van der Kallen's group structure on orbit sets of unimodular rows \cite{vdKallen83}.  In fact, it is possible to use the formula in Construction~\ref{construction:geometriccomposition} to establish directly that the composition so-defined is unital, associative and commutative.  Following this course likely produces a composition in greater generality than we have established here (e.g., presumably one could make a statement, suitably modified, that holds for $U$ the spectrum of an arbitrary commutative Noetherian ring).  Nevertheless, we have not pursued this approach because, from a theoretical point of view, we felt the homotopy theoretic techniques in Section~\ref{s:cohomotopy} give a nice ``explanation" for the formulas regarding composition and yield strong functoriality properties without significant additional work.
\end{rem}

\begin{rem}
\label{rem:vdk}
Assume $k$ is a field.  The variety $Q_{2n-1}$ is $\aone$-weakly equivalent to ${\mathbb A}^n \setminus 0$ and therefore $\aone$-$(n-2)$-connected.  As a consequence, if $X = \Spec R$ is a smooth $k$-scheme of dimension $d \leq 2n-4$, we conclude that $[X,Q_{2n-1}]_{\aone}$ inherits a group structure via Proposition \ref{prop:groupstructureI}.  By Theorem~\ref{thm:naivevsrefined}, we know that $\pi_0(\Singaone Q_{2n-1})(R) \to [X,Q_{2n-1}]_{\aone}$ is a bijection for any smooth affine $k$-scheme $X$.  On the other hand, by \cite[Theorem 2.1]{Fasel08b} we know that $\pi_0(\Singaone Q_{2n-1})(R) \cong Um_n(R)/E_n(R)$ and therefore the latter is equipped with an abelian group structure.  Now, van der Kallen showed $Um_n(R)/E_n(R)$ admits an abelian group structure \cite[Theorem 5.3]{vdKallen89}.  In fact, the two group structures actually coincide, and we will return to this in future work.
\end{rem}


\section{Segre classes and Euler class groups}
\label{s:Euler}
In this section, we connect the results of the previous two sections with classical results involving Euler class groups.  Section~\ref{ss:Segreclasshomomorphism} is devoted to construction of the homomorphism from Euler class groups to motivic cohomotopy; the main result is Theorem~\ref{thm:comparison}.  In Section~\ref{ss:applications} we establish some applications of this comparison result: we compare Euler class groups, weak Euler class groups, Chow--Witt groups and Chow groups in certain situations.

\subsection{Euler class groups and the Segre class homomorphism}
\label{ss:Segreclasshomomorphism}
We begin by recalling some key results of \cite{Fasel15b} related to Segre classes.  Then, we record the definition of Euler class groups \`a la Bhatwadekar--Sridharan, slightly adapted for our purposes in Definition~\ref{defn:eulerclassgroup}.  After recording a helpful moving lemma, we construct the Segre class homomorphism and then establish its main properties.  The main result of this section is Theorem \ref{thm:comparison}, which establishes a portion of Theorem \ref{thmintro:cohomotopy} in the introduction.

\subsubsection*{Universal Segre classes}
Once more suppose $R$ is a commutative ring and $f: \Spec R \to Q_{2n}$ is a morphism given by a sequence of elements $(x,y,z) \in R^n \times R^n \times R$.  If $I$ is the ideal $\langle x_1,\ldots,x_n,z \rangle$, then the quotient $I/I^2$ is generated by $\{\overline{x}_1,\ldots,\overline{x}_n\}$ and there is a surjective homomorphism $\omega_I:(R/I)^n\to I/I^2$.

\begin{defn}
\label{defn:segreclasses}
If $R$ is a commutative ring and $n\in\N$, then set
\begin{enumerate}[noitemsep,topsep=1pt]
\item $Ob_n(R) := $ the set consisting of pairs $(I,\omega_I)$ where $I$ a finitely generated ideal in $R$ and $\omega_I:(R/I)^n\to I/I^2$ a surjective homomorphism;
\item $Ob_n'(R) :=$ the subset of $Ob_n(R)$ consisting of pairs $(I,\omega_I)$ with $\operatorname{ht}(I) = n$.
\end{enumerate}
\end{defn}

We now review a key result from \cite{Fasel15b}; we include the proof for the convenience of the reader.

\begin{thm}
\label{thm:segre}
Suppose $(I,\omega_I) \in Ob_n(R)$.
\begin{enumerate}[noitemsep,topsep=1pt]
\item There exist elements $a_1,\ldots,a_n,s \in I$ and $b_1,\ldots,b_n \in R$ such that $I = \langle a,s \rangle$ and $(a,b,s) \in Q_{2n}(R)$.
\item If $(a',b',s') \in R^{n} \times R^n \times R$ is any another $(2n+1)-$tuple of elements such that $a_1',\ldots,a_n',s' \in I$, $I = \langle a',s' \rangle$ and $(a',b',s') \in Q_{2n}(R)$, then $[(a,b,s)] = [(a',b',s')]$ in $\pi_0(\Singaone Q_{2n}(R))$, i.e., the class $[(a,b,s)]$ is independent of the relevant choices.
\end{enumerate}
\end{thm}

\begin{proof}
The first statement is immediate from Lemma~\ref{lem:liftexists}.  The second statement follows from Lemma~\ref{lem:naivehomotopyclassdependsonideal} by analyzing the proof of Lemma~\ref{lem:liftexists}.
\end{proof}

\begin{defn}
\label{defn:segreclass}
Given $(I,\omega_I)\in Ob_n(R)$, set $s(I,\omega_I) = [(a,b,s)] \in \pi_0(\Singaone Q_{2n}(R))$ where $a_1,\ldots,a_n,s \in I$, $b_1,\ldots,b_n \in R$, $I = \langle a,s \rangle$ and $(a,b,s) \in Q_{2n}(R)$ (guaranteed to exist by \textup{Theorem~\ref{thm:segre}}.  The class $s(I,\omega_I)$ will be called the \emph{universal Segre class} of $(I,\omega_I)$.
\end{defn}

\begin{rem}
The terminology Segre class is inspired by the work of Murthy \cite[\S 5]{Murthy94}.
\end{rem}

\subsubsection*{Euler class groups}
Let $R$ be a commutative Noetherian ring with $\mathrm{dim}(R)=d$.  Euler class groups of $R$ were defined for height $d$ ideals in \cite[p. 197]{Bhatwadekar00} and more generally in \cite[p. 146]{Bhatwadekar02}.  We slightly recast the definition here and to do so, we recall some notation.  Given a pair $(I,\omega_I) \in Ob_n(R)$, if $I$ is generated by $a_1,\ldots,a_n$, we say that {\em $\omega_I$ is induced by $\overline{a}_1,\ldots,\overline{a}_n$} if $\omega_I$ is the morphism sending the standard basis vector $e_i$ of the free $R/I$-module $R/I^n$ to the element $\overline{a}_i$.  For any commutative ring $A$, let $E_n(A) \subset GL_n(A)$ be the subgroup consisting of elementary (shearing) matrices.  Note that there is a left action of $E_n(R/I)$ on $Ob_n(R)$ or $Ob_n'(R)$ given as follows: for an element $\sigma \in E_n(R/I)$, $\sigma \cdot (I,\omega_I) = (I,\omega_I \circ \sigma^{-1})$.

\begin{defn}
\label{defn:eulerclassgroup}
The Euler class group $\mathrm{E}^n(R)$ is the quotient of the free abelian group $\Z[Ob^{\prime}_n(R)]$ by the ideal generated by the following relations:
\begin{enumerate}[noitemsep,topsep=1pt]
\item (Complete intersection) If $I=\langle a_1,\ldots,a_n\rangle$ and $\omega_I: (R/I)^n \to I/I^2$ is induced by $\overline a_1,\ldots,\overline a_n$ then $(I,\omega_I)=0$.
\item (Elementary action) If $\sigma \in E_n(R/I)$, then $\sigma \cdot (I,\omega_I)= (I,\omega_I)$.
\item (Disconnected sum) If $I=JK$, where $J,K$ are height $n$ ideals with $K+J=R$, then a surjection $\omega_I:(R/I)^n\to I/I^2$ induces surjections $\omega_K:(R/K)^n\to K/K^2$ and $\omega_J:(R/J)^n\to J/J^2$ and the relation is $(I,\omega_I)=(K,\omega_K)+(J,\omega_J)$.
\end{enumerate}
\end{defn}

\begin{rem}
\label{rem:eulerclassgroups}
Definition~\ref{defn:eulerclassgroup} is equivalent to that given in \cite[p. 147]{Bhatwadekar02} even though it looks (slightly) formally different.  More precisely, Bhatwadekar and Sridharan consider the free abelian group on equivalence class of pairs $(I,\omega_I)$ with $\Spec R/I$ connected; conditions (2) and (3) are imposed precisely to compare with that situation.

Note also that Definition~\ref{defn:eulerclassgroup} always makes sense, but it is mainly of interest when $d \leq 2n-3$ because in that context it is closely related with with the problem of when the ideal $I$ can be generated by $n$ elements (\cite[Theorem 4.2]{Bhatwadekar02} or \cite[Theorem 2.4]{Mandal12}).
\end{rem}

\subsubsection*{A ``moving" lemma and some consequences}
We now establish some preliminary ``moving" results, which will be useful in getting ``good" representatives of elements of Euler class groups.  The following lemma is a special case of \cite[Lemma 2.2]{Mandal12} (taking $J = I^2$ in the notation there); note that the proof is quite similar to that of Lemma~\ref{lem:moving}.

\begin{lem}
\label{lem:movingeuler}
Let $R$ be a Noetherian ring of dimension $d$, $I \subset R$ an ideal of height $n$ and $\omega_I:R^n\to I/I^2$ a surjective homomorphism.  For $I_1,\ldots,I_r$ arbitrary ideals of $R$, there exists an ideal $K\subset R$ and a homomorphism $f:R^n\to I\cap K$ having the following properties:
\begin{enumerate}[noitemsep,topsep=1pt]
\item the ideals $I^2$ and $K$ are comaximal;
\item the composite $R^n\stackrel{f}\to I\cap K\subset I\to I/I^2$ is equal to $\omega_I$;
\item the ideal $K$ has height $\geq n$;
\item for any integer $1\leq i\leq r$, the inequality $\mathrm{ht}((I_i+K)/I_i)\geq n$ holds.
\end{enumerate}
\end{lem}

The following result can be deduced from Lemma~\ref{lem:movingeuler}.  It also follows by combining \cite[Corollary 2.4 and Proposition 3.1]{Bhatwadekar02}; see \cite[p. 66]{vdKallen15} and note that \cite[Corollary 2.4]{Bhatwadekar02} only requires $d \leq 2n-1$.

\begin{cor}
\label{cor:alphaisI}
Let $R$ be a Noetherian ring of dimension $d\leq 2n-1$ and let $\alpha\in \mathrm{E}^n(R)$. There exist an ideal $I$ of height $n$ and a surjective homomorphism $\omega_I:(R/I)^n\to I/I^2$ such that $\alpha=(I,\omega_I)\in \mathrm{E}^n(R)$.
\end{cor}

\subsubsection*{The Segre class homomorphism}
Assume $k$ is a field having characteristic not equal to $2$, and suppose $X=\Spec R$ is a smooth affine $k$-scheme of dimension $d\leq 2n-2$.  By definition of the Segre class (Definition~\ref{defn:segreclass}), a pair $(I,\omega_I)\in Ob'_n(R)$ determines an element $s(I,\omega_I)$ of $\pi_0(\Singaone Q_{2n})(R)$ so by Theorem~\ref{thm:naivevsrefined} yields an element of $[X,Q_{2n}]_{\aone}$.  If we equip the target with the group structure of Proposition~\ref{prop:groupstructureI}, then there is a unique extension to a group homomorphism:
\[
\tilde{s}: \Z[Ob'_n(R)] \longrightarrow [X,Q_{2n}]_{\aone}.
\]
We now show that $\tilde{s}$ factors through the Euler class group of Definition~\ref{defn:eulerclassgroup}.

\begin{prop}
\label{prop:segreclasshomomorphism}
Suppose $n,d$ are integers, $n \geq 2$ and $d \leq 2n-2$.  Assume $k$ is a field having characteristic not equal to $2$ and suppose $X = \Spec R$ is a smooth affine $k$-scheme having $\aone$-cohomological dimension $d \leq 2n-2$.  The homomorphism $\tilde{s}$ factors through a homomorphism:
\[
s:\mathrm{E}^n(R) \longrightarrow [X,Q_{2n}]_{\aone},
\]
that we will refer to as the Segre class homomorphism.
\end{prop}

\begin{proof}
We prove that the relations in Definition~\ref{defn:eulerclassgroup} are satisfied in $[X,Q_{2n}]_{\aone}$.
\smallskip

\noindent {\bf Step 1}.  Relation $1$ holds, i.e., if $I=\langle a_1,\ldots,a_n\rangle$ is an ideal of height $n$ and $\omega_I:(R/I)^n\to I/I^2$ is given by $e_i\mapsto \overline a_i$, then $s(I,\omega_I)=0$.  Indeed, the Segre class of $(I,\omega_I)$ is given, for instance, by $v:=(a_1,\ldots,a_n,0,\ldots,0)\in Q_{2n}(R)$. Now,  $(a_1T,\ldots,a_nT,0,\ldots,0)\in Q_{2n}(R[T])$ provides an explicit homotopy between $v$ and $v_0=(0,\ldots,0)$.
\smallskip

\noindent{\bf Step 2.} Relation $2$ holds, i.e., if $(I,\omega_I)$ is a generator of the Euler class group and $\sigma\in E_{n}(R/I)$, then $s(I,\omega_I)=s(\sigma \cdot (I,\omega_I))$.  By definition, any element $\sigma \in E_n(R/I)$ can be factored as a product of elementary matrices.  Therefore, it suffices to establish the result for $\overline{\sigma} = 1 + e_{ij}(\overline{\lambda})$, $i \neq j$ with $\lambda \in R$ an arbitrary element.

Choose any representative $(a,b,s)$ of $(I,\omega_I)$; such a representative exists by Lemma~\ref{lem:liftexists}.  In that case, set $\sigma = 1 + e_{ij}\lambda$ and observe that we can choose the representative $(a\cdot \sigma,b\cdot(\sigma^{-1})^t ,s)$ for $\sigma \cdot (I,\omega_I)$.  Thus, $\sigma(T) := 1 + e_{ij}(\lambda T)$ determines an explicit homotopy between $(a,b,s)$ and $(a\cdot \sigma,b\cdot(\sigma^{-1})^t,s)$.
\smallskip

\noindent{\bf Step 3.} Relation $3$ holds, i.e., suppose $I,J,K$ are ideals of height $n$ such that $J$ and $K$ are comaximal, $I = JK$, $\omega_I: (R/I)^n \to I/I^2$ is a surjection, and $\omega_J$ and $\omega_K$ are the surjections induced by $\omega_I$; we will show that $s(I,\omega_I) = s(J,\omega_J) + s(K,\omega_K)$.  To this end, recall Construction~\ref{construction:geometriccomposition}: if $J$ and $K$ correspond to elements $v,v^\prime \in Q_{2n}(R)$, we observed there how to construct an element of $Q_{2n}(R)$ corresponding with $I$.  With that in mind, relation 3 follows immediately from Theorem~\ref{thm:concretetau}.
\end{proof}

Our goal is to show that the Segre class homomorphism is an isomorphism; the next result shows that it is surjective and establishes the third point of Theorem~\ref{thmintro:cohomotopy}.

\begin{prop}
\label{prop:surjectivityofsegreclasshomom}
Assume $k$ is a field having characteristic not equal to $2$, and suppose $n,d$ are integers, $n \geq 2$ and $d \leq 2n-2$.  If $X = \Spec R$ is a smooth affine $k$-scheme having $\aone$-cohomological dimension $d$, then the Segre class homomorphism $s$ is surjective.
\end{prop}

\begin{proof}
Fix $v=(a,b,s)\in Q_{2n}(R)$.  By Lemma~\ref{lem:moving} we can find $v^\prime=(a^\prime,b^\prime,s^\prime)$ such that (i) $I:=\langle a^\prime,s^\prime\rangle$ has height $\geq n$ and (ii) the class of $v^{\prime}$ in $\pi_0(\Singaone Q_{2n}(R))$ coincides with that of $v$.  Since $a^\prime (b^\prime)^t=s^\prime(1-s^\prime)$, we conclude that the localization $I_{1-s^\prime}=\langle a^\prime\rangle$.  If $I_{1-s^{\prime}}$ is a proper ideal of $R_{1-s^{\prime}}$, it follows from Krull's Hauptidealsatz that $\mathrm{ht}(I_{1-s^\prime})\leq n$ and therefore that $\mathrm{ht}(I)\leq n$ as well.  Thus, in that case, we conclude that $\mathrm{ht}(I) = n$.  Defining $\omega_I:(R/I)^n\to I/I^2$ by $e_i\mapsto \overline {a_i}^\prime$, it follows that $s(I,\omega_I)=(a^\prime,b^\prime,s^\prime)$. On the other hand, if $I_{1-s^\prime}=R_{1-s^\prime}$, then $I=R$.  Thus, $(a',b',s')$ corresponds to the constant map to the base-point in $\pi_0(\Singaone Q_{2n}(R))$ (i.e., the identity for the group structure). Since the images of $(a,b,s)$ and $(a',b',s')$ in $\pi_0(\Singaone Q_{2n}(R))$ coincide, the result follows in this case as well.
\end{proof}

\subsubsection*{An inverse to the Segre class homomorphism}
To prove injectivity of the Segre class homomorphism, we will construct an explicit inverse.  As above, let $n \geq 2$ be an integer, and suppose $X=\Spec R$ be a smooth affine $k$-scheme having $\aone$-cohomological dimension $d \leq 2n-2$.  Given a sequence $v = (a,b,s) \in Q_{2n}(R)$, Lemma~\ref{lem:moving} guarantees that we may find $\mu \in R^n$ such that, setting
\[
w = (a + \mu(1-s)^2,b(1-\mu b^t),s+\mu b^t (1-s)),
\]
the following statements are true: $w \in Q_{2n}(R)$, $w$ and $v$ lie in the same naive $\aone$-homotopy class and the ideal $N_{\mu} := \langle a + \mu(1-s)^2,s+\mu b^t (1-s) \rangle$ has height $\geq n$.  Mapping the standard basis of $R^n$ to the classes of $a_i + \mu_i(1-s)^2$ in the quotient $N_{\mu}/N_{\mu}^2$ defines a surjective homomorphism $\omega_{N_{\mu}}: R^n \to N_{\mu}/N_{\mu}^2$.  Granted these facts, define a map
\[
\tilde{\psi}_n: Q_{2n}(R) \longrightarrow \mathrm{E}^n(R)
\]
by means of the formula
\[
\tilde{\psi}_n(v) = (N_{\mu},\omega_{N_{\mu}}).
\]
We now check that this formula is well-defined, i.e., independent of the choice of $\mu$.

\begin{lem}
\label{lem:psin}
Assume $k$ is an infinite field and $R$ is a smooth affine $k$-algebra of dimension $d \geq 2$.  If $n \geq 2$ is an integer, then the map $\tilde{\psi}_n:Q_{2n}(R)\to \mathrm{E}^n(R)$ described above is well-defined.
\end{lem}

\begin{proof}
Given $v = (a,b,s)$, assume we may find $\mu_0$ and $\mu_1$ as guaranteed by Lemma~\ref{lem:moving} such that, setting
\[
\begin{split}
w_0 &:=(a+\mu_0 (1-s)^2,b(1-\mu_0 b^t),s+\mu_0 b^t(1-s)), \text{ and } \\
w_1 &:=(a+\mu_1 (1-s)^2,b(1-\mu_1 b^t),s+\mu_1 b^t(1-s)),
\end{split}
\]
both $w_0$ and $w_1$ satisfy the conditions listed above.

We now construct a homotopy between the elements $w_0$ and $w_1$.  To this end, set $\beta = \mu_0(1-T) + \mu_1 T$.  One checks immediately that
\[
W := (a + \beta(1-s)^2,b(1 - \beta b^t), s + \beta b^t(1-s))
\]
defines an element of $Q_{2n}(R[T])$ with $W(0) = w_0$ and $W(1) = w_1$.  A priori, the ideal $N_{\beta} := \langle a + \beta(1-s)^2,s + \beta b^t(1-s) \rangle$ may not give rise to an element of $\mathrm{E}^n(R[T])$, so we now show that it may be modified to do so.

Begin by observing that the sequence $(a+\beta(1-s)^2,T(1-T)(1-s)^2)$ defines a unimodular row over $R[T]^{n+1}_{T(1-T)(1-s)}$.  Arguing as in Lemma~\ref{lem:moving}, we conclude that there exists $\alpha \in R[T]^n$ such that the ideal
\[
J :=\langle a+\beta (1-s)^2,\alpha T(1-T)(1-s)^2\rangle
\]
has height $\geq n$ in $R[T]_{T(1-T)(1-s)}$.  Setting $\gamma=\beta +\alpha T(1-T)$, we see that
\[
W^\prime =(a+\gamma (1-s)^2,b(1-\gamma b^t),s+\gamma b^t(1-s))\in Q_{2n}(R[T]).
\]
still satisfies $W^\prime(0)=w_0$ and $W^\prime(1)=w_1$.  We claim, furthermore, that the ideal
\[
I :=\langle a+\gamma (1-s)^2,s+\gamma b^t(1-s)\rangle
\]
has height $\geq n$ in $R[T]$.  Indeed, suppose that $I \subset \mathfrak p$ for some prime ideal ${\mathfrak p} \subset R[T]$. If $T(1-T)(1-s)\not\in\mathfrak p$, then by definition $\mathrm{ht}(\mathfrak p)\geq n$. If $T\in \mathfrak p$, we find $\langle I ,T\rangle\subset \mathfrak p$. Now $\langle I ,T\rangle=\langle N_{\mu},T\rangle$ and it follows that $\mathrm{ht}(\mathfrak p)\geq n+1$.  The same argument applies if $(1-T)\in \mathfrak p$ and we similarly conclude that $\mathrm{ht}(\mathfrak{p}) \geq n+1$.  Finally, $\langle I,1-s\rangle=R[T]$ and the claim is proved.

Assume henceforth that $k$ is an infinite field and $R$ is a smooth affine $k$-algebra of dimension $d \geq 2$.  The pair $(I,\omega_{I})$ gives rise to an element of $\mathrm{E}^n(R[T])$.  The image of this element under the maps $\mathrm{E}^n(R[T]) \to \mathrm{E}^n(R)$ corresponding to restriction to $t = 0$ and $t = 1$ correspond to the elements of the Euler class group defined by $w_0$ and $w_1$.  However, under the hypothesis on $d$, the restriction maps are isomorphisms by appeal to Theorem~\ref{homotopy}.
\end{proof}


Now that we have established that $\tilde\psi_n$ defines a function $Q_{2n}(R) \to \mathrm{E}^n(R)$, we want to check that it factors through the naive $\aone$-homotopy relation.

\begin{prop}
\label{prop:inverse}
Assume $k$ is an infinite field and $R$ is a smooth affine $k$-algebra of dimension $d \geq 2$.  If $n \geq 2$ is an integer, then the function $\tilde{\psi}_n:Q_{2n}(R)\to \mathrm{E}^n(R)$ factors through the naive $\aone$-homotopy relation to define a function $\psi_n: [X,Q_{2n}]_{\aone}\to \mathrm{E}^n(R)$.
\end{prop}

\begin{proof}
Let $v =(A,B,S)\in Q_{2n}(R[T])$. By Lemma~\ref{lem:moving}, there exists $\mu \in R[T]^n$ such that
\[
w =(A+\mu(T)(1-s(T))^2,B(1-\mu(T)B^t),S+\mu B^t(1-S))\in Q_{2n}(R[T])
\]
satisfies, $\mathrm{ht}(N)\geq n$ where $N:=\langle A+\mu(1-S)^2,S+\mu B^t(1-S)\rangle$. Further, we may assume (arguing as in the proof of Lemma~\ref{lem:movingeuler}) that $\mathrm{ht}(N(0))\geq n$ and $\mathrm{ht}(N(1))\geq n$. It follows from the previous lemma that $\tilde{\psi}_n(v(0))=(N(0),\omega_{N(0)})$ and $\tilde\psi_n(v(1))=(N(1),\omega_{N(1)})$.  As in the proof of Lemma~\ref{lem:psin}, we conclude by appeal to Theorem~\ref{homotopy}.
\end{proof}

We now put everything together to show that the Segre class homomorphism is an isomorphism; this establishes the fourth point of Theorem~\ref{thmintro:cohomotopy}.

\begin{thm}
\label{thm:comparison}
If $k$ is an infinite field having characteristic not equal to $2$, $n,d$ are integers $\geq 2$, and $R$ is a smooth affine $k$-algebra of dimension $d \leq 2n-2$, then the Segre class homomorphism
\[
s:\mathrm{E}^n(R) \longrightarrow [X,Q_{2n}]_{\aone}
\]
is an isomorphism.
\end{thm}

\begin{proof}
Surjectivity of the Segre class homomorphism was already established in Proposition~\ref{prop:surjectivityofsegreclasshomom}, so it remains to demonstrate injectivity.  To this end, we use the map $\psi_n$ of Proposition~\ref{prop:inverse}. Indeed, it follows from Lemma~\ref{lem:psin} that $(\psi_n\circ s)(I,\omega_I)=(I,\omega_I)$ for any generator $(I,\omega_I)$ of $E^n(R)$ and injectivity thus follows from Corollary~\ref{cor:alphaisI}.
\end{proof}

\begin{rem}
\label{rem:assumptions}
Tracing through the arguments, in Theorem~\ref{thm:comparison}, the assumption that $k$ has characteristic not equal to $2$ appears only by way of appeal to Theorem~\ref{thm:naivevsrefined}.  In particular, work currently in preparation will remove this hypothesis.  On the other hand, the assumption that $k$ is infinite appears by way of appeal to Theorem~\ref{homotopy}, i.e., homotopy invariance for Euler class groups.  It is unclear to the authors at the time of this writing whether this hypothesis can be removed.
\end{rem}

\begin{rem}
\label{rem:improvedfunctoriality}
Theorem~\ref{thm:comparison} immediately implies that Euler class groups satisfy a number of functorial properties.  For example, if $f: X \to Y$ is any arbitrary morphism of smooth affine $k$-varieties as in the statement, then the map $f^*: [Y,Q_{2n}]_{\aone} \to [X,Q_{2n}]$ induced by composition yields a pull-back morphism for Euler class groups.  By appeal to Theorem~\ref{thm:excision} one obtains Mayer-Vietoris sequences.  More precisely, if $X$ is a smooth affine scheme of dimension $d \leq 2n-2$ and we have $j: U \to X$ an open immersion of an affine $k$-scheme $U$ and $\varphi: V \to X$ an \'etale morphism from an affine $k$-scheme $V$ such that the induced map $(V \setminus U \times_X V)_{red} \to (X \setminus U)_{red}$ is an isomorphism, then there is an exact sequence of the form
\[
[X,Q_{2n}]_{\aone} \longrightarrow [U,Q_{2n}]_{\aone} \times [V,Q_{2n}]_{\aone} \longrightarrow [U \times_X V,Q_{2n}]_{\aone}.
\]
Furthermore, if $d \leq 2n-4$, there is a map from $[U \times_X V,Q_{2n-1}]$ to $[X,Q_{2n}]$ which extends the exact sequence further to the left.  Finally, Remark~\ref{rem:externalproducts} equips Euler class groups with a product operation as well. It would be interesting to compare the excision and product operations studied in \cite{Mandal12} and \cite{Mandal10} with this product operation.
\end{rem}

\begin{rem}
As a further consequence of Theorem~\ref{thm:comparison}, given a ring $R$ satisfying the hypotheses, it is possible to attach to a pair $(I,\omega_I)$ consisting of an ideal $I \subset R$ and a surjection $\omega_I: (R/I)^n \to I/I^2$ an element in the Euler class group that detects if $\omega_I$ lifts to a surjection $R^n \to I$.  In particular, one can extend (and partially generalize) the work of M. Das and R. Sridharan \cite{Das10}.
\end{rem}

\subsection{Euler class, Chow--Witt and Chow groups}
\label{ss:applications}
Finally, we put everything together: if $X$ is a smooth $k$-scheme of dimension $d$, we compare Euler class groups and Chow--Witt groups in top codimension.  The explicit form of the Hurewicz homomorphism from Theorem~\ref{thm:abelian} plays a role in the comparison with previous constructions.  In Theorem~\ref{thm:EulertoChow} we show that Euler class groups can be identified with Chow--Witt groups in certain cases.  Finally, Theorem~\ref{thm:weakEulerisChow} establishes the connection between weak Euler class groups and Chow groups by appeal to an exact sequence studied in \cite{Das15}.

\subsubsection*{Euler class groups vs. Chow--Witt groups}
In view of Theorems~\ref{thm:abelian} and \ref{thm:comparison}, there is a homomorphism $\mathrm{E}^n(R)\to \CH^n(X)$ which we make more explicit when $n=d:=\mathrm{dim}(X)$. In that case, the Euler class group is generated by pairs $(\mathfrak m,\omega_{\mathfrak m})$ where $\mathfrak m\subset R$ is a maximal ideal and $\omega_{\mathfrak m}:(R/\mathfrak m)^d\to \mathfrak m/\mathfrak m^2$ is a surjection (and indeed an isomorphism).  Note that this definition of the Euler class differs from that in Definition~\ref{defn:eulerclassgroup}, but if we restrict attention to smooth affine algebras over an infinite perfect field, it follows from \cite[Remark 4.6]{Bhatwadekar98} that the definition is equivalent.

Let $m_1,\ldots,m_n$ be elements of $\mathfrak m$ such that $\omega_{\mathfrak m}(e_i)=\overline m_i$.  In that case, the description of the Hurewicz homomorphism from Theorem~\ref{thm:abelian} shows that the Hurewicz image of $(s(\mathfrak m,\omega_{\mathfrak m}))$ is given (up to the unit $\lambda\in \K_0^{MW}(k)$) by the class of the cycle $\langle 1\rangle\otimes \overline m_1\wedge\ldots\wedge\overline m_d$. In other words, the composite homomorphism
\begin{equation}
\label{eqn:eulertochowwitt}
\xymatrix{\mathrm{E}^d(R)\ar[r]^-s & [X,Q_{2d}]\ar[r] & \CH^d(X)}
\end{equation}
coincides (up to the unit $\lambda\in \K_0^{MW}(k)$) with the homomorphism defined in \cite[Proposition 17.2.8]{Fasel08a}.  Using these observations, we establish the following result, which proves another part of Theorem \ref{thmintro:cohomotopy} from the introduction. 

\begin{thm}
\label{thm:EulertoChow}
If $k$ is an infinite field that has characteristic not equal to $2$, $d \geq 2$ is an integer and $X = \Spec R$ is a smooth affine $k$-scheme of dimension $d$, then the composite homomorphism
\[
\mathrm{E}^d(R)\to \CH^d(X)
\]
in \eqref{eqn:eulertochowwitt} is an isomorphism.
\end{thm}

\begin{proof}
It suffices to show that the composite
\[
\xymatrix{\mathrm{E}^d(R)\ar[r]^-s & [X,Q_{2d}]_{\aone}\ar[r] & \CH^d(X)}
\]
is an isomorphism.  In view of Theorem~\ref{thm:abelian}, this follows from Theorem~\ref{thm:comparison}.
\end{proof}

\subsubsection*{Weak Euler class groups and Chow groups: recollections}
Let $R$ be a smooth affine algebra of dimension $d$ over a field $k$.  Recall the notion of weak Euler class groups $\mathrm{E}^d_0(R)$ (cf. \cite[Definition 2.2]{Bhatwadekar99} or \cite[Definition 5.1]{Murthy97}; these groups were mentioned in the introduction).

\begin{defn}
\label{defn:weakeulerclassgroup}
If $X = \Spec R$ is a smooth affine $k$-scheme of dimension $d$, the weak Euler class group $\mathrm{E}^d_0(X)$ is the quotient of the free abelian group on the set of maximal ideals $\mathfrak m\subset R$ subject to the relation $\sum_i \mathfrak m_i=0$ if $I=\cap_i \mathfrak m_i$ is a reduced complete intersection ideal.
\end{defn}

The assignment $({\mathfrak m},\omega_{\mathfrak m}) \mapsto {\mathfrak m}$ passes to a surjective group homomorphism $\mathrm{E}^d(X) \to \mathrm{E}^d_0(X)$.  Associating with a maximal ideal $\mathfrak m$ its class in the Chow group $CH^d(X)$ yields a well-defined surjective homomorphism $s^\prime:\mathrm{E}^d_0(R)\to CH^d(X)$ \cite[Lemma 2.5]{Bhatwadekar99}.  The relationship between these homomorphisms together with the canonical homomorphism $\CH^d(X) \to CH^d(X)$ is established in \cite[Proposition 17.2.10]{Fasel08a} where it is shown that there is a commutative diagram of the form:
\[
\xymatrix{\mathrm{E}^d(X)\ar[r]\ar[d]_-s & \mathrm{E}^d_0(X)\ar[d]^-{s^\prime} \\
\CH^d(X)\ar[r] & CH^d(X)};
\]
and the horizontal maps are surjective (the surjectivity of the bottom horizontal map follows because $X$ has dimension $d$).  We now proceed to analyze the kernels of the two horizontal maps appearing here.

\subsubsection*{On the surjection $\mathrm{E}^d(X) \to \mathrm{E}^d_0(X)$}
The kernel of the surjection $\mathrm{E}^d(X) \to \mathrm{E}^d_0(X)$ can also be described in terms of unimodular rows.  Recall that $Um_{d+1}(R)$ is the set of unimodular rows of length $d+1$ in $R$, i.e., rows $(a_1,\ldots,a_{d+1})$ that admit a right inverse.

In Section~\ref{ss:dzhomomorphism}, following \cite{Das15} one defines a function
\[
\phi: Um_{d+1}(R)/E_{d+1}(R) \longrightarrow \mathrm{E}^d(R)
\]
that is a homomorphism by appeal to Theorem~\ref{prop:dzhomomrevisited}.


\begin{prop}[Das--Zinna]
\label{prop:daszinna}
Assume $k$ is an infinite field, and suppose $X = \Spec R$ is a smooth affine $k$-algebra of dimension $d \geq 2$.  The homomorphism $\phi: Um_{d+1}(X)/E_{d+1}(X) \to \mathrm{E}^d(X)$ described above fits into an exact sequence of the form
\[
Um_{d+1}(X)/E_{d+1}(X) \stackrel{\phi}{\longrightarrow} \mathrm{E}^d(X) \longrightarrow \mathrm{E}^d_0(X) \longrightarrow 0.
\]
\end{prop}

\begin{proof}
Appealing to Proposition~\ref{prop:dzhomomrevisited} instead of the construction and results of \cite[\S 3]{Das15}; the claimed exactness follows by repeated the proof of \cite[Theorem 3.8]{Das15}.
\end{proof}

\subsubsection*{On the surjection $\CH^d(X) \to CH^d(X)$}
The kernel of the surjection $\CH^d(X) \to CH^d(X)$ can be analyzed homotopically.  The following result generalizes \cite[Theorem 4.9]{Fasel08b}.

\begin{prop}
\label{prop:hurewiczisomorphismadminus0}
Assume $k$ is a field, $d \geq 2$ is an integer and $X$ is a smooth affine $k$-scheme of dimension $d$.  There is a ``Hurewicz" isomorphism of the form:
\[
Um_{d+1}(X)/E_{d+1}(X) \isomto H^d(X,\K^{MW}_{d+1}).
\]
Explicitly, this isomorphism sends a unimodular row $(a_1,\ldots,a_{d+1})$ to the cycle $[a_{d+1}] \tensor \overline{a_1} \wedge \cdots \wedge \overline{a_d}$ in $\K^{MW}_1(k(s)) \tensor_{\Z[k(s)]^{\times}} \Z[\Lambda^*_s]$, where $s$ is the complete intersection given by the equations $a_1 = \cdots = a_{d} = 0$.
\end{prop}

\begin{proof}
By appeal to \cite[Corollary 4.2.6]{AsokHoyoisWendtII} and \cite[Theorem 2.1]{Fasel08b}, one knows that $[X,{\mathbb A}^{d+1} \setminus 0]_{\aone} \cong Um_{d+1}(X)/E_{d+1}(X)$.  Since ${\mathbb A}^{d+1} \setminus 0$ is $\aone$-$(d-1)$-connected, and  $\bpi_d^{\aone}({\mathbb A}^{d+1} \setminus 0) \cong \K^{MW}_{d+1}$, by appeal to Proposition~\ref{prop:functorialityinthetarget}, we conclude that for any integer $d \geq 2$ the Hurewicz homomorphism defines an isomorphism $[X,{\mathbb A}^{d+1} \setminus 0]_{\aone} \isomt H^d(X,\K^{MW}_{d+1})$.

The final statement follows from the explicit description of the fundamental class in $H^d({\mathbb A}^{d+1} \setminus 0,\K^{MW}_{d+1})$ from \cite[Lemma 4.5]{ADF}.
\end{proof}

Taking cohomology of the exact sequence of sheaves $0 \to \mathbf{I}^{d+1} \to \K^{MW}_d \to \K^M_d \to 0$ yields an exact sequence of the form
\[
H^d(X,\mathbf{I}^{d+1}) \longrightarrow \CH^d(X) \longrightarrow CH^d(X) \longrightarrow 0.
\]
The epimorphism $\K^{MW}_{d+1} \to \mathbf{I}^{d+1}$ yields an surjective homomorphism $H^d(X,\K^{MW}_{d+1}) \to H^d(X,\mathbf{I}^{d+1})$ and we thus obtain a homomorphism $H^d(X,\K^{MW}_{d+1}) \to \CH^d(X)$; note that this homomorphism may also be described as the map in sheaf cohomology associated with the morphism of sheaves $\K^{MW}_{d+1} \to \K^{MW}_d$.

Define $\phi'$ to be the composite
\[
Um_{d+1}(X)/E_{d+1}(X) \isomto H^d(X,\K^{MW}_{d+1}) \longrightarrow \CH^d(X),
\]
where the first map is the isomorphism from Proposition~\ref{prop:hurewiczisomorphismadminus0} and the second map $H^d(X,\K^{MW}_{d+1}) \to \CH^d(X)$ is the one just defined.  Putting everything together, one has the following result.

\begin{prop}
\label{prop:kernelofchowwitttochow}
Assume $k$ is a field, and $X$ is a smooth affine $k$-scheme of dimension $d \geq 2$.  There is an exact sequence of the form
\[
Um_{d+1}(X)/E_{d+1}(X) \stackrel{\phi'}{\longrightarrow} \CH^d(X) \longrightarrow CH^d(X) \longrightarrow 0.
\]
\end{prop}

\subsubsection*{Comparing weak Euler class groups and Chow groups}
We now establish Theorem~\ref{thmintro:mainchow} from the introduction.

\begin{thm}
\label{thm:weakEulerisChow}
If $k$ is an infinite field that has characteristic not equal to $2$, and $X=\Spec R$ is a smooth affine $k$-scheme of dimension $d\geq 2$, then the homomorphism $s^\prime:\mathrm{E}^d_0(R)\to CH^d(X)$ is an isomorphism.
\end{thm}

\begin{proof}
We claim that $s\phi = \phi'$.  Assuming this, by combining Propositions~\ref{prop:daszinna}, \ref{prop:kernelofchowwitttochow} and \cite[Proposition 17.2.10]{Fasel08a} one obtains a commutative diagram with exact rows of the form:
\[
\xymatrix{Um_{d+1}(R)/E_{d+1}(R)\ar[r]^-\phi \ar@{=}[d]& \mathrm{E}^d(R)\ar[r]\ar[d]_-s & \mathrm{E}^d_0(R)\ar[d]^-{s^\prime}\ar[r] & 0 \\
Um_{d+1}(R)/E_{d+1}(R)\ar[r]_-{\phi^\prime} & \CH^d(X)\ar[r] & CH^d(X)\ar[r] & 0.}
\]
The conclusion then follows by appeal to Theorem~\ref{thm:EulertoChow}.

To establish that $s\phi = \phi'$, it suffices to unwind the construction of $\phi'$.  Indeed, $\phi'$ is the composite of the Hurewicz isomorphism, which is made explicit in Proposition~\ref{prop:hurewiczisomorphismadminus0} and the map in sheaf cohomology induced by $\K^{MW}_{d+1} \to \K^{MW}_d$.  The latter map is induced by an explicit morphism of complexes, which we now describe.

Suppose $F$ is a field.   The homomorphism $K^{MW}_1(F) \to K^{MW}_0(F)$ is given at the level of symbols by multiplication by the class $\eta$.  In more detail, if $u$ is a unit in $F$, then
\[
[u] \longmapsto \eta[u] = 1 + \eta[u] - 1 = \langle u \rangle - \langle 1 \rangle.
\]
Now, we use the notation of Proposition~\ref{prop:hurewiczisomorphismadminus0}.  In that notation, the composite of the Hurewicz homomorphism and the map on sheaf cohomology sends the unimodular row $(a_1,\ldots,a_{d+1})$ to $(\langle a_{d+1}\rangle - \langle 1 \rangle) \tensor \bar{a}_1 \wedge \cdots \wedge\bar{a_d}$.  By means of the explicit formula for $\phi$ and the discussion of $s$ before Theorem~\ref{thm:EulertoChow} (specifically the fact that it coincides with the Hurewicz homomorphism for $Q_{2d}$), the result follows.
\end{proof}


\appendix
\section{Some results on Euler class groups}
\begin{center}by Mrinal Kanti Das\end{center}
\label{s:properties}
This appendix attempts to streamline (e.g., clarify necessary hypotheses) the presentation of some results on Euler class groups used in the main body of this paper.  Results established in this appendix are often not presented in their most general form; our purpose is solely to tailor them for use in the main body of the paper.  Moreover, the proofs given herein are not original: most ideas are present either explicitly or implicitly in the literature.  The results presented here incorporate a portion of the document \cite{das2019remarks} and provide further weakenings of the hypotheses for results listed there.

\subsection{Homotopy invariance of Euler class groups} 
\label{ss:homotopyinvarianceofEulerclassgroups}
Let $R$ be a Noetherian commutative ring of dimension $d\geq 2$.  Let $I\subset R[T]$ be an ideal of height $n$ such that $I/I^2$ is generated by $n$ elements and both $I(0)$ and $I(1)$ are ideals of height $n$ in $R$. If $\omega:(R[T]/I)^n \twoheadrightarrow I/I^2$ is a surjection, then there are induced surjections $\omega(0):(R/I(0))^n \twoheadrightarrow I(0)/I(0)^2$ and $\omega(1):(R/I(1))^n \twoheadrightarrow I(1)/I(1)^2$. One says that {\em homotopy invariance} holds for the Euler class group of $R$ if $(I(0),\omega(0))=(I(1),\omega(1))$ in $\mathrm{E}^n(R)$ (see Section~\ref{ss:Segreclasshomomorphism}).  

An example due to Bhatwadekar \cite[Example 5.21]{dk} shows that homotopy invariance fails for Euler class in general.  Nevertheless, under suitable smoothness hypotheses, homotopy invariance is known to hold (see, e.g., \cite[Corollary 4.4(3)]{Mandal12} and \cite[Proposition 5.7]{Das03}).  In this section, we establish Theorem~\ref{homotopy} which improves those results.  To this end, we begin by recalling the following improvement of \cite[Theorem 3.8]{Bhatwadekar98} from \cite{d3}.


\begin{thm}[{\cite[Theorem 4.12]{d3}}]
	\label{nori}
Let $R$ be a regular ring of dimension $d$ which has essentially finite type over a field 
$k$ and infinite residue fields.  Let $n$ be an integer such that $2n\geq d+3$.  If $I\subset R[T]$ is an ideal of height $n$ such that $I=(F_1,\cdots,F_n)+(I^2T)$, then there
exist $G_1,\cdots,G_n\in I$ such that $I=(G_1,\cdots,G_n)$ where $G_i-F_i\in (I^2T)$ for $1\leq i\leq n$.
\end{thm}

\begin{rem}
The above result fails in the case where $d=n=2$, as shown in \cite[Example 3.15]{Bhatwadekar98}.  However, in that situation, the following weaker assertion holds (see also \cite[Theorem 3.16]{Bhatwadekar98}.
\end{rem}

\begin{thm}[{\cite[Corollary 3.5]{d3}}]
\label{nori2}
Let $R$ be a regular ring of dimension $2$.  If $I\subset R[T]$ is an ideal of height $2$ such that $\text{ht}\,I(0)\geq 2$ and $I=(F_1,F_2)+(I^2T)$, then there exist $H_1,H_2\in I$ and $\theta\in SL_2(R[T]/I)$  such that the following statements hold:
\begin{enumerate}[noitemsep,topsep=1pt]
\item $I=(H_1,H_2)$; 
\item $(\overline{F_1},\overline{F_2})\theta= (\overline{H_1},\overline{H_2})$, where bar denotes modulo $I^2$; and 
\item $H_i(0)=F_i(0)$ for $i=1,2$. 
\end{enumerate}
\end{thm}

We now prove the following result; the idea of this proof is similar to that of \cite[Lemma 4.3, Remark 4.6]{Bhatwadekar98}.

\begin{thm}
\label{homotopy}
Let $R$ be a ring of dimension $d\geq 2$ and suppose $n \geq 2$ is a positive integer. Assume further that: (1) if $d=n=2$, then $R$ is regular or (2) if $2n-3 \geq d \geq 3$, then $R$ is a smooth $k$-algebra with $k$ an infinite field. Let $I\subset R[T]$ be an ideal of height $n$ such that $I/I^2$ is generated by $n$ elements and both $I(0)$ and $I(1)$ are ideals of height $n$ in $R$. If $\omega:(R[T]/I)^n \twoheadrightarrow I/I^2$ is a surjection, then  $(I(0),\omega(0))=(I(1),\omega(1))$ in $\mathrm{E}^n(R)$.	
\end{thm}

\begin{proof}
Let $\omega$ be induced by $I=(f_1,\cdots,f_n)+I^2$. Then $\omega(0)$ is given by
$I(0)=(f_1(0),\cdots,f_n(0))+I(0)^2$. Similarly, $\omega(1)$ is given by
$I(1)=(f_1(1),\cdots,f_n(1))+I(1)^2$.

We first assume that $(I(0),\omega(0))\neq 0$ in $E^n(R)$.
Let $I\cap R=J$. Applying \cite[Corollary 2.4]{brs4} we can find an ideal $K$ of height $n$ 
such that $K+J=R$ and $I(0)\cap K=(a_1,\cdots,a_n)$ where
$a_i-f_i(0)\in I(0)^2$. Then $K= (a_1,\cdots,a_n)+K^2$. Write $\omega_K:(R/K)^n \twoheadrightarrow K/K^2$ for the surjection thus induced. We then have
\[
(I(0),\omega(0))+(K,\omega_K)=0 \text{ in } \mathrm{E}^n(R).
\]

Let $L=I\cap K[T]$. Then $f_1,\cdots,f_n$ and $a_1,\cdots,a_n$ will induce a set of generators of   
$L/L^2$, say,  $L=(g_1,\cdots,g_n)+L^2$. This means that $g_i-f_i\in I^2$
and $g_i-a_i\in K^2[T]$. Consequently,
$g_i(0)-f_i(0)\in I(0)^2$ and $g_i(0)-a_i\in K^2$. Since $a_i-f_i(0)\in I(0)^2$, it follows that 
$g_i(0)-a_i\in (I(0)\cap K)^2=L(0)^2$. Applying \cite[Remark 3.9]{Bhatwadekar98} we can find $h_1,\cdots,h_n\in L$ such that  $L=(h_1,\cdots,h_n)+(L^2T)$ with $h_i-g_i\in L^2$. \newline

\noindent {\bf Case 1.}
Assume that  $2n-3\geq d\geq 3$, then by \ref{nori} we have $L=(l_1,\cdots,l_n)$ such that $l_i-h_i\in (L^2T)$.  Then $L(1)=I(1)\cap K=(l_1(1),\cdots,l_n(1))$. As $l_i-f_i\in I^2$, we have $l_i(1)-f_i(1)\in I(1)^2$. Also, $l_i-a_i\in K^2[T]$, implying that $l_i(1)-a_i\in K^2$. Summing up, we simply have 
\[
(I(1),\omega(1))+(K,\omega_K)=0 \text{ in } E^n(R).
\]
Therefore, $(I(0),\omega(0))=(I(1),\omega(1))$ in $E^n(R)$ in this case. \newline

\noindent {\bf Case 2.} Assume that $d=n=2$. By \ref{nori2}, there exist $l_1,l_2\in L$ and $\theta\in SL_2(R[T]/L)$  such that (1) $L=(l_1,l_2)$; (2) $(\overline{h_1},\overline{h_2})\theta= (\overline{l_1},\overline{l_2})$,
where we use the bar to indicate that we are working modulo $L^2$. 

We have $L(1)=I(1)\cap K=(l_1(1),l_2(1))$. Note that $\theta(1)\in SL_2(R/I(1)\cap K)$ and as 
$\mathrm{dim}(R/I(1)\cap K)=0$, we have $\theta(1)\in E_2(R/I(1)\cap K)=E_2(R/I(1))\times E_2(R/K)$. 

As $I(1)+K=R$ and $I(1)\cap K=(l_1(1),l_2(1))$, we have, $I(1)=(l_1(1),l_2(1))+I(1)^2$ and
$K=(l_1(1),l_2(1))+K^2$. Writing the corresponding surjections as $\omega'_{(I(1))}$ and 
$\omega'_K$ we have $(I(1),\omega'_{(I(1))})+(K,\omega'_K)=0$.

It is then easy to see that $\omega(1)$ and $\omega'_{(I(1))}$ differ by a matrix in $E_2(R/I(1))$, whereas $\omega_K$ and $\omega'_K$ differ by a matrix in $E_2(R/K)$. Consequently, $(I(1),\omega(1))=(I(1),\omega'_{(I(1))})$ and $(K,\omega'_K)=(K,\omega_K)$ and we have $(I(1),\omega(1))+(K,\omega_K)=0$. Therefore, $(I(0),\omega(0))=(I(1),\omega(1))$ in $\mathrm{E}^2(R)$.
\end{proof}

\subsection{A group homomorphism}
\label{ss:dzhomomorphism}
Assume $R$ is a commutative ring of dimension $d\geq 2$.  Our goal is to construct a group homomorphism $Um_{d+1}(R)/E_{d+1}(R)$ to $\mathrm{E}^d(R)$, following \cite{Das15}, with some subtle modifications.  Another construction of this homomorphism is given (under rather general hypotheses) in \cite{vdKallen15}, but we choose this approach because it may be plugged directly into the proof of Proposition~\ref{prop:daszinna}.  We refer the reader to \cite{Das15} and \cite{vdKallen15} for further references and some history of the relevant homomorphism.  

Our construction proceeds by defining a function on a restricted class of unimodular rows, checking invariance under the action of the elementary matrix group, and then showing that the map is a group homomorphism.  To this end, let $(a_1,\cdots,a_d,a_{d+1})\in Um_{d+1}(R)$.  By adding suitable multiples of $a_{d+1}$ to $a_1,\cdots,a_d$ if necessary, we can  assume that  $\mathrm{ht}(a_1,\cdots,a_d)=d$.  Let $J_0$ be the ideal $\langle a_1,\cdots,a_d \rangle$ and $\omega_0:R^d\twoheadrightarrow J_0$ be the surjection induced by $(a_1,\cdots,a_d)$. As $a_{d+1}$ is a unit modulo $J_0$, the pair $(J_0,\overline{a_{d+1}}\omega_0)$ defines an element of $Ob'_{d}(R)$ 
(where "bar" means reduction modulo $J_0$) and has an associated class in $\mathrm{E}^d(R)$. We associate with $(a_1,\cdots,a_d,a_{d+1})$ the element $(J_0,\overline{a_{d+1}}\omega_0)\in \mathrm{E}^d(R)$.  The next proposition follows \cite[Section 4]{Das15}; for the sake of completeness, we spell out the details.

\begin{prop}
Assume $R$ is a commutative ring of dimension $d\geq 2$.  Assume furthermore that: (1) $d = 2$ and $R$ is regular, or (2) $d \geq 3$ and $R$ is a smooth $k$-algebra with $k$ an infinite field.  The association described above is invariant under the action of the elementary subgroup $E_{d+1}(R)$ and therefore yields a well-defined function: 
\[
\phi:Um_{d+1}(R)/E_{d+1}(R) \longrightarrow E^d(R).
\]
\end{prop}

\begin{proof}
Let $(a_1,\ldots,a_d,a_{d+1}),(b_1,\ldots,b_d,b_{d+1})\in Um_{d+1}(R)$ be such that 
$\mathrm{ht}(a_1,\ldots,a_d)=d=\mathrm{ht}(b_1,\ldots,b_d)$ and there exists 
$\sigma\in E_{d+1}(R)$ such that  
\[
(a_1,\ldots,a_d,a_{d+1})\sigma=(b_1,\ldots,b_d,b_{d+1}).
\]
Let $J_1= \langle b_1,\ldots,b_d \rangle$ and $\omega_1:R^d\twoheadrightarrow J_1$ be the surjection induced by $(b_1,\ldots,b_d)$.  It suffices to prove that $(J_0,\overline{a_{d+1}}\omega_0)=(J_1,\overline{b_{d+1}}\omega_1)$ in $\mathrm{E}^d(R)$ (where $\overline{b_{d+1}}$ is the reduction of $b_{d+1}$ modulo $J_1$).

As $\sigma$ is elementary, it is naively homotopic to the identity, i.e., we may find $\tau\in E_{d+1}(R[T])$ such that $\tau(0)=\text{id}$ and $\tau(1)=\sigma$. Let
\[
(a_1,\ldots,a_d,a_{d+1})\tau=(f_1(T),\ldots,f_d(T),f_{d+1}(T)).
\]
Then $(f_1(T),\ldots,f_d(T),f_{d+1}(T))\in Um_{d+1}(R[T])$.  As $J_0$ and $J_1$ are both of height $d$, one checks that the ideal $(f_1(T), \ldots, f_d(T),T(1-T))\subset R[T]$  has height $d+1$.  Since $(f_1(T), \ldots , f_d(T),f_{d+1}(T))$ is a unimodular row, we conclude that
\[
\mathrm{ht}(f_1(T), \ldots, f_d(T),(T(1-T)f_{d+1}(T))=d+1.
\]
Adding suitable multiples of $(T(1 -T))f_{d+1}(T) $ to $f_1(T),\ldots,f_d(T)$ if necessary, we can assume that $\mathrm{ht}(f_1(T),\ldots,f_d(T))=d$. Write $I=\langle f_1(T),\ldots,f_d(T) \rangle$.
Then we have   $I(0) = J_0$ and $I(1) = J_1$.  
Let $\omega^{\prime}:R[T]^d \twoheadrightarrow I$ denote the surjection induced by $f_1(T),\ldots,f_d(T)$ and write $\omega=\overline{f_{d+1}}(T)\omega^{\prime}$  (where $\overline{f_{d+1}}$ is the reduction of $f_d$ modulo $I$). Then, $(I(0),\omega(0))=(J_0,\overline{a_{d+1}}\omega_0)$ and
$(I(1),\omega(1))=(J_1,\widetilde{b_{d+1}}\omega_1)$ in $\mathrm{E}^d(R)$.
Appealing to Theorem~\ref{homotopy}, we conclude that 
$(J_0,\overline{a_{d+1}}\omega_0)=(J_1,\widetilde{b_{d+1}}\omega_1)$ in $\mathrm{E}^d(R)$ as required.
\end{proof}

\begin{prop}
\label{prop:dzhomomrevisited}
The function $\phi:Um_{d+1}(R)/E_{d+1}(R)\to E^d(R)$ constructed above is a group homomorphism.
\end{prop}

\begin{proof}
By \cite[3.3]{vdKallen02}, it suffices to prove that if $(x,a_1,\cdots,a_d)$ and $(y,a_1,\cdots,a_d)$ are unimodular with $x+y=1$, then 
\[
\phi(x,a_1,\cdots,a_d)+\phi(y,a_1,\cdots,a_d)=\phi(xy,a_1,\cdots,a_d).
\]
If any of $x,y$ or $xy$ is zero, then  $(a_1,\cdots,a_d)$ is unimodular and the above equality holds trivially; thus, we may assume that these elements are non-zero.

Let bar denote reduction modulo $\langle xy \rangle$.  Adding suitable multiples of $\overline{a_d}$ to $ \overline{a_1}, \cdots , \overline{a_{d-1}}$ we may assume that $\mathrm{ht}(\overline{a_1}, \cdots , \overline{a_{d-1}})\geq d-1$. It follows that $\mathrm{ht}(xy,a_1,\cdots,a_{d-1})\geq d$. Therefore, $\mathrm{ht}(x,a_1,\cdots,a_{d-1})\geq d$ and $\mathrm{ht}(y,a_1,\cdots,a_{d-1})\geq d.$

Set $J_1=(x,a_1,\cdots,a_{d-1})$ and  $J_2=(y,a_1,\cdots,a_{d-1})$ and let $\omega_{J_1}:R^d \twoheadrightarrow J_1$ and $\omega_{J_2}:R^d \twoheadrightarrow J_2$ be the corresponding surjections. Then $J_1+J_2=R$ and $J_1\cap J_2=(xy,a_1,\cdots,a_{d-1})$; let $\omega_{J_1\cap J_2}$ be the corresponding surjection.

Unwinding the definition of $\phi$, the following statements hold:
\begin{enumerate}[noitemsep,topsep=1pt]
	\item
	$\phi [x,a_1,\cdots,a_{d}]=(J_1,\overline{a_d}\omega_{J_1})$
	\item
	$\phi [y,a_1,\cdots,a_{d}]=(J_2,\overline{a_d}\omega_{J_2})$
	\item
	$\phi [xy,a_1,\cdots,a_{d}]=(J_1\cap J_2,\overline{a_d}\omega_{J_1\cap J_2})$
\end{enumerate}
(note that $\overline{a_d}$ refers to reduction modulo $J_1$ in the first line, reduction modulo $J_2$ in the second line and reduction modulo $J_1 \cap J_2$ in the third line).  As $x\equiv 1$ modulo $J_2$, and $y\equiv 1$ modulo $J_1$, we see that 
\[
(J_1\cap J_2,\overline{a_d} \omega_{J_1\cap J_2})=(J_1,\overline{a_d}\omega_{J_1})+(J_2,\overline{a_d}\omega_{J_2}),
\]
and the result follows. 
\end{proof}


\begin{footnotesize}
\bibliographystyle{alpha}
\bibliography{complete}
\end{footnotesize}
\Addresses
\end{document}